\documentclass{article}

\usepackage{multicol}
\usepackage{amsmath,amsfonts,amsthm,amssymb}
\usepackage{graphicx}
\usepackage{tikz}
\usepackage{enumitem}
\usepackage{verbatim}
\usepackage{color}
\usepackage{wasysym}
\usepackage{floatrow}
\usepackage[normalem]{ulem}
\usepackage{setspace}
% ...
\usepackage{soul} 

\newtheorem{thm}{Theorem}

\newtheorem{cnj}[thm]{Conjecture}
\newtheorem{cor}[thm]{Corollary}
\newtheorem{fct}[thm]{Fact}
\newtheorem{lem}[thm]{Lemma}

\def\a{{\alpha}}
\def\b{{\beta}}

\def\d{{\delta}}
\def\D{{\Delta}}
\def\e{{\epsilon}}
\def\ep{{e_\pi}}

\def\i{{\iota}}
\def\k{{\kappa}}
\def\l{{\lambda}}

\def\p{{\pi}}
\def\P{{\Pi}}
\def\r{{\rho}}
\def\s{{\sigma}}

\def\W{{\Omega}}

\def\cA{{\cal A}}

\def\cL{{\cal L}}

\def\cP{{\cal P}}

\def\cS{{\cal S}}

\def\sP{{\sf P}}

\def\zN{{\mathbb N}}

\def\mfs{{\mathfrak{s}}}

\def\dC{{\dot{C}}}

\def\dD{{\dot{D}}}
\def\dF{{\dot{F}}}

\def\ol{\overline{\ell}}

\def\cost{{\sf cost}}
\def\diam{{\sf diam}}
\def\dist{{\sf dist}}
\def\ecc{{\sf ecc}}

\def\len{{\sf length}}
\def\pot{{\sf pot}}

\def\supp{{\sf supp}}

\def\Clj{{C'_{L_j}}}

\def\Cljs{{C_{L_j}^*}}

\def\Crj{{C'_{R_j}}}

\def\Crjs{{C_{R_j}^*}}

\def\Dlj{{D_{L_j}}}

\def\Drj{{D_{R_j}}}

\def\ltnk{{l_t(n,k)}}
\def\Ltnk{{L_{t,n,k}}}
\def\ptnk{{p_t(n,k)}}
\def\Pnk{{P_n^{(k)}}}
\def\wtn{{w_t(n)}}
\def\Wtn{{W_{t,n}}}

\def\Aaij{{\cA_{\langle i,j\rangle}}}

\def\Fij{{F_{[i,j]}}}
\def\Faij{{F_{\langle i,j\rangle}}}
\def\Gij{{G_{[i,j]}}}
 
\def\mt{{\emptyset}}
\def\pf{{\hfill $\Box$\bigskip}}
\def\peb{{\mapsto}}
\def\rar{{\rightarrow}}
\def\sse{{\subseteq}}

\definecolor{brwn}{RGB}{140, 70, 20}
\definecolor{gren}{RGB}{  0,140, 10}

\newcommand{\la}[1]{\textcolor{blue}{\sf{#1}}}
\newcommand{\gh}[1]{\textcolor{brwn}{\sf{#1}}}
\newcommand{\god}[1]{\textcolor{red}{\sf{#1}}}

\pgfmathsetmacro\ang{30}
\pgfmathsetmacro\Ang{45}
\pgfmathsetmacro\ANg{60}
\pgfmathsetmacro\verts{12}
\pgfmathsetmacro\Verts{-1+\verts}
\pgfmathsetmacro\VErts{-2+\verts}
\pgfmathsetmacro\VERts{-3+\verts}
\pgfmathsetmacro\VERTs{-4+\verts}

\begin{document}

\title{Pebbling in powers of paths} 

\author{
Liliana Alc\' on \thanks{
CeMaLP, UNLP, CONICET, Argentina
}
\and
Glenn Hurlbert \thanks{
Department of Mathematics and Applied Mathematics,
Virginia Commonwealth University, USA
\newline
}
}
\date{\today}
\maketitle

\centerline{\it Dedicated to the memory of Ron Graham.}

\begin{abstract}
The $t$-{\it fold pebbling number}, $\p_t(G)$, of a graph $G$ is defined to be the minimum number $m$ so that, from any given configuration of $m$ pebbles on the vertices of $G$, it is possible to place at least $t$ pebbles on any specified vertex via pebbling moves.
It has been conjectured that the pebbling numbers of pyramid-free chordal graphs can be calculated in polynomial time.

The $k^{\rm th}$ power $G^{(k)}$ of the graph $G$ is obtained from $G$ by adding an edge between any two vertices of distance at most $k$ from each other.
The $k^{\rm th}$ power of the path $P_n$ on $n$ vertices is an important class of pyramid-free chordal graphs, and is a stepping stone to the more general class of $k$-paths and the still more general class of interval graphs.
Pachter, Snevily, and Voxman (1995) calculated $\pi(P_n^{(2)})$, Kim (2004) calculated $\pi(P_n^{(3)})$, and Kim and Kim (2010) calculated $\pi(P_n^{(4)})$.
In this paper we calculate $\pi_t(P_n^{(k)})$ for all $n$, $k$, and $t$.

For a function $D:V(G)\rar \zN$, the $D$-{\it pebbling number}, $\p(G,D)$, of a graph $G$ is defined to be the minimum number $m$ so that, from any given configuration of $m$ pebbles on the vertices of $G$, it is possible to place at least $D(v)$ pebbles on each vertex $v$ via pebbling moves.
It has been conjectured that $\p(G,D)\le\p_{|D|}(G)$ for all $G$ and $D$.
We make the stronger conjecture that every $G$ and $D$ satisfies $\p(G,D)\le \p_{|D|}(G)-(s(D)-1)$, where $s(D)$ counts the number of vertices $v$ with $D(v)>0$.
We prove that trees and $\Pnk$, for all $n$ and $k$,
satisfy the stronger conjecture.

The {\it pebbling exponent} $e_\pi(G)$ of a graph $G$ was defined by Pachter, et al., to be the minimum $k$ for which $\pi(G^{(k)})=n(G^{(k)})$.
Of course, $e_\pi(G)\le \diam(G)$, and Czygrinow, Hurlbert, Kierstead, and Trotter (2002) proved that almost all graphs $G$ have $e_\pi(G)=1$.
Lourdusamy and Mathivanan (2015) proved several results on $\pi_t(C_n^2)$, and Hurlbert (2017) proved an asymptotically tight formula for $e_\pi(C_n)$.
Our formula for $\pi_t(P_n^{(k)})$ allows us to compute $e_\pi(P_n)$ within additively narrow bounds.
\end{abstract}

\section{Introduction}
\label{s:intro}

Graph pebbling has an interesting history, with many challenging open problems.
Calculating pebbling numbers of graphs is a well known computationally difficult problem (in $\P_2^{\sP}$-complete \cite{MilCla}).
See \cite{HurlGenGP,HurlHand} for more background.
It has been asked (e.g. \cite{HurlModMeth}) for what families of graphs $G$ the pebbling number $\pi(G)$ (defined below) can be calculated in polynomial time.
One possible family posited in \cite{AGHSplit} is that of chordal graphs, most likely with some restriction, such as bounded diameter or treewidth, for example.
This paper follows a sequence (\cite{AGHSplit,AGH2P,AGHSemi2T,AGHkConnU}) intended to provide at least a partial answer to this line of inquiry, which has led us to make the following conjecture.
We define the {\it pyramid} to be the triangulated 6-cycle $abcdef$ with interior triangle $bdf$, and say that a graph is $H$-{\it free} if it does not contain $H$ as an induced subgraph.

\begin{cnj}
\label{j:chordal}
If $G$ is a pyramid-free chordal graph then $\p(G)$ can be calculated in polynomial time.
\end{cnj}

A {\it configuration} C of pebbles on the vertices of a connected graph G is a function $C : V(G)\rar\zN$ (the nonnegative integers), so that $C(v)$ counts the number of pebbles placed on the vertex $v$. 
A vertex $v$ is {\it empty} if $C(v)=0$ and {\it big} if $C(v)\ge 2$.
We write $|C|$ for the {\it size} $\sum_v C(v)$ of $C$; i.e. the number of pebbles in the configuration. 
A {\it pebbling step} from a big vertex $u$ to one of its neighbors $v$ (denoted $u\peb v$) reduces $C(u)$ by two and increases $C(v)$ by one. 
Given a specified {\it target} vertex $r$ we say that $C$ is $t$-{\it fold} $r$-{\it solvable} if some sequence of pebbling steps places $t$ pebbles on $r$. 
We are concerned with determining $\pi_t(G,r)$, the minimum positive integer  $m$ such that every configuration of size $m$ on the vertices of $G$ is $t$-fold $r$-solvable. 
The $t$-{\it pebbling number} of $G$ is defined to be $\pi_t(G) = \max_{r\in  V(G)} \pi_t(G,r)$.
We adopt the natural interpretation that $\p_0(G)=0$, and avoid writing $t$ when $t=1$.

The $k^{\rm th}$ power $G^{(k)}$ of the graph $G$ is obtained from $G$ by adding an edge between any two vertices of distance at most $k$ from each other.
The {\it pebbling exponent} $\ep(G)$ of a graph $G$ was defined in \cite{PaSnVo} to be the minimum $k$ for which $\pi(G^{(k)})=n(G^{(k)})$.
For example, {\it Class 0} graphs (graphs $G$ with $\p(G)=n(G)$) have pebbling exponent $\ep(G)=1$.
In a very strong probabilistic sense (see \cite{CHKT}) almost all graphs $G$ have $\ep(G)=1$.
Of course, $\ep(G)\le\diam(G)$, and the authors of \cite{PaSnVo} ask specifically about the cycle $C_n$ on $n$ vertices. 
In \cite{HurlWFL} it was shown that $n/(2\lg n)\le \ep(C_n)\le n/[2(\lg n-\lg\lg n)]$, which determines its exact value for $n\le 9$, bounds it within a factor of two always, and a factor of one asymptotically.
Here we write $\lg$ for the base 2 logarithm. Lourdusamy and Mathivanan \cite{LourdMath} proved several results on $\p_t(C_n^2)$.

Denote the path on $n$ vertices by $P_n$.
Pachter, et al. \cite{PaSnVo} proved that $\p(P_n^{(2)})=2^{\lceil\frac{n-1}{2}\rceil}+((n-2)\bmod 2)$ for $n\ge 2$.
Kim \cite{Kim} proved that $\p(P_n^{(3)})=2^{\lceil\frac{n-1}{3}\rceil}+((n-2)\bmod 3)$ for $n\ge 8$ (and equals $n$ for $n\le 7$).
Kim and Kim \cite{KimKim} proved that $\p(P_n^{(4)})=2^{\lceil\frac{n-1}{4}\rceil}+((n-2)\bmod 4)$ for $n\ge 14$ (and equals $n$ for $n\le 13$).
We generalize these results in Theorem \ref{t:main}, below, calculating $\p_t(\Pnk)$ for all $k$ and all $t$.
This allows us to compute $e_\p(P_n)$ very tightly in Corollary \ref{c:expo}.

We generalize the traditional pebbling model as follows.
A {\it pebbling function} $F$ is any function $F:V\rar\zN$; its {\it size} is $|F|=\sum_{v\in V}F(v)$.
For a pebbling function $F$, define $\dF$ to be the multiset $\{v^{F(v)}\}_{v\in V}$
(the exponent $F(v)$ is the {\it multiplicity} of $v$).
Configurations $C$ (of pebbles) and distributions $D$ (of targets) are both pebbling functions.
However, we think of $\dC$ as a multiset of pebbles, labeled by their vertex locations, while we think of $\dD$ as a multiset of target vertices.
Furthermore, we think of $v_{i,j}$ as the label of $j^{\rm th}$ pebble (or target) sitting on vertex $v_i$.
For $m\in\zN$ and $x\in V$, define the function $mx$ by $mx(v)= m$ if $v=x$ and 0 otherwise.
Thus, the symbol $x$ can represent a vertex or a pebbling function, depending on its context.
In particular, if $F=mx$ then $\dF=\{x^m\}$.

For a configuration $C$ and distribution $D$, we say that $C$ is $D$-{\it solvable} (or that there is a $(C,D)$-{\it solution}, or that $G$ has a $(C,D)$-{\it solution}) if some sequence of pebbling steps places at least $D(v)$ pebbles on each vertex $v$.
The $D$-{\it pebbling number}, $\p(G,D)$, of a graph $G$ is defined to be the minimum number $m$ such that $G$ is $(C,D)$-solvable whenever $|C|\ge m$.
Thus we can write $\p_t(G,r)$ as $\p(G,tr)$ in this generalized notation.
(We note that the $D$-pebbling number was first introduced in \cite{CCFHPST} for the case $D=V(G)$, and was called the {\it cover pebbling number}.)

Just as the $t$-fold pebbling number can be used inductively to prove results about the pebbling number (e.g. \cite{Chung,AGHSemi2T}), the $D$-pebbling number can be used inductively to prove results about the $t$-fold pebbling number.
It is also thought that this might be a powerful tool in attacking Graham's Conjecture on the pebbling number of the cartesian product of graphs (see \cite{Chung}).
The following {\bf Weak Target Conjecture} was conjectured in \cite{HerHesHur}.

\begin{cnj}
\label{j:tTargetWeak}
\cite{HerHesHur}
Every graph $G$ satisfies $\p(G,D)\le \p_{|D|}(G)$ for every target distribution $D$.
\end{cnj}

The authors of \cite{HerHesHur} verified this conjecture for trees, cycles, complete graphs, and cubes, and the authors of \cite{HurlSedd} verified this conjecture for 2-paths and Kneser graphs $K(m,2)$ with $m\ge 5$.

Define $\supp(D)$ to be the set of vertices $v$ with $D(v)>0$, and denote $s(D)=|\supp(D)|$.
We make the following {\bf Strong Target Conjecture}.

\begin{cnj}
\label{j:tTargetStrong}
Every graph $G$ satisfies $\p(G,D)\le \p_{|D|}(G)-(s(D)-1)$ for every target distribution $D$.
\end{cnj}

We prove that trees satisfy this stronger conjecture in Theorem \ref{t:tree}.
We prove in Theorem \ref{t:Dsolve} that $\Pnk$ satisfies this stronger conjecture for all $n$ and $k$.

Another reason to study pebbling on powers of paths is the following.
Czygrinow, et al. \cite{CHKT}, proved that, for each $d\ge 1$, there is 
a least positive integer $k(d)$ such that if $G$ has diameter $d$ and
connectivity at least $k(d)$ then $G$ is Class 0.
They showed that $k(d)\le 2^{2d+3}$ and $k(d)\in\W(2^d/d)$.
We note that $\Pnk$ has connectivity equal to $k$ and that Theorem \ref{t:main} shows that $\p(\Pnk)=n$ when $k\ge (2^d-2)/(d-1)$.
Thus $\Pnk$ witnesses the tightness of the lower bound on $k(d)$; we believe that the upper bound on $k(d)$ is weak and should be improved.
Furthermore, at the other extreme when $t$ is large, a theorem of \cite{HMOZ} states that every graph $G$ satisfies $\lim_{t\rar\infty}\p_t(G)/t=2^d$, where $d=\diam(G)$.
Our Theorem \ref{t:main} formula for $t\ge k(d-1)/(2^d-2)$ is more precise, that $\p_t(\Pnk)=t2^d+((n-2)\bmod{k})$.

The final and, for our current purposes, most important motivation for investigating $\p_t(\Pnk)$ for $2\le k\le\diam(\Pnk)=\lfloor(n-2)/k \rfloor+1$ is that $\Pnk$ is chordal.
This case is a key stepping stone toward the graph classes $k$-paths and interval graphs.
At each stage of the sequence of papers \cite{AGHSplit,AGH2P,AGHSemi2T,AGHkConnU} mentioned above we have discovered new hurdles that have required new techniques which have expanded our understanding of pebbling in chordal graphs, such as the technical lemmas found in Section \ref{s:tech}.
The critical pieces of the puzzle in this paper are the new chordal lemmas found in Section \ref{ss:chordal}, as well as the careful interplay between the $t$-wide and $t$-long cases in the inductive proof of Section \ref{s:proof_D_solve}.
In particular, Conjecture \ref{j:tTargetStrong} plays a crucial role and may be the most important contribution of this work, as a powerful technique in future research.

We describe our results in the next section, introduce the important machinery in Section \ref{s:tech}, and prove Theorems \ref{t:tree}--\ref{t:main} in order in Sections \ref{s:trees}--\ref{s:thm}.

\section{Main Theorems} 
\label{s:def}

For  positive integers $n$ and $k$, the {\it path power} $P^{(k)}_n$ is the graph  with vertex set $V=\{v_1,...v_{n}\}$, and $v_i\sim v_j$ whenever $1\le |i-j|\le k$.
The diameter $d$ of $P_n^{(k)}$ is completely determined by $n$ and $k$; in fact, \[d=\lfloor(n-2)/k \rfloor+1.\]
We let $b=(n-2)-k(d-1)$, then  $0\leq b <k$. 
For $n\geq 2$ and $t\geq 1$, define the following functions, where the above formula for $d$ is assumed:
\begin{center}
\begin{tabular}{rcl}
     $\ltnk$ & $=$ & $t2^d+n-2-k(d-1)=t2^d+b$,  \\
     $\wtn$&$ =$&$2t+n-2$,\\
    $p_t(n,k)$ & $=$  & $ \max\{\ltnk,\wtn\}$.
\end{tabular}
\end{center}

We have $\wtn\ge \ltnk$ if and only if 
$t(2^d-2)\le (d-1)k$; 
in other words   \[p_t(n,k)= \left\{
  \begin{array}{ll}
     2t+n-2 & \hbox{if } t\leq t_0, \hbox{ and}\\
    t2^d+b  & \hbox{if } t\geq t_0,
  \end{array}
\right.\]
where $t_0=1$ if $d=1$, and $t_0=t_0(k,d)=k(d-1)/(2^d-2)$ otherwise.
We say that $\Pnk$ is $t$-{\it wide} when $t\leq t_0$ and $t$-{\it long} when $t\geq t_0$.
More finely, we say that $\Pnk$ is {\it barely} $t$-long when $t_0<t=\lceil t_0\rceil$, and {\it strictly} $t$-long if $t>\lceil t_0\rceil$.
Thus, the formula names $w_t$ and $l_t$ match the $t$-wide and $t$-long terminology.
We will occasionally make use of the fact that, for $d\ge 2$ and fixed $k$, $t_0(k,d)$ is a strictly decreasing function of $d$.

For $k\geq n-1$, $P_n^{(k)}$ is a complete graph, and it is $t$-wide for any $t\geq 1$. Therefore, $\pi_t(\Pnk,r)= \pi_t(\Pnk)=2t+n-2=\wtn=p_t(n,k)$ for any vertex $r$.  
Notice that in this case all vertices are  {\it simplicial} (i.e. their neighborhoods induce complete graphs).
For $1\leq k < n-1$, the vertices $v_1$ and $v_n$ are the only two  simplicial vertices of $P^{(k)}_n$; the non simplicial vertices will be called {\it interior} vertices.  Note that the connectivity of $P_n^{(k)}$ equals $k$.
Moreover, when $1 \leq (n-1)/2 \le k$, $\Pnk$ contains a universal vertex (adjacent to every other vertex); then the present work extends that of \cite{AGHkConnU}, which computes the pebbling numbers of $k$-connected graphs with universal vertices.
For practical purposes, when $n=1$,  we define  $p_t(1,k)=t-1$ for any $k$.

The next three theorems will be proved in Sections \ref{s:trees}--\ref{s:thm}.

The following theorem verifies the Strong Target Conjecture 
for trees. The cost of a solution refers to the total number of pebbles that are lost when performing the steps of the solution, the formal definition is given in the next section.

\begin{thm}
\label{t:tree}
Let $T$ be a tree of diameter $d$ and $D$ be a target distribution of size $t$.
Then $\p(T,D)\le\p_t(T)-s(D)+1$.
Furthermore, if $C$ is a configuration on $T$ of size $|C|\ge \p_t(T)-s(D)+1$, then $C$ solves any target $v\in\dD$ with cost at most $2^{\ecc_T(v)}$.
\end{thm}

The following theorem verifies the Strong Target Conjecture 
for powers of paths.

\begin{thm}
\label{t:Dsolve}
Let $D$ be a $t$-multiset of target vertices of $\Pnk$.
Then $\p(\Pnk,D)\le\ptnk-s(D)+1$.
\end{thm}

Analogous to Theorem \ref{t:tree}, as part of the proof of Theorem \ref{t:Dsolve}, we show in the $t$-long case that if $C$ is a configuration on $G=\Pnk$ of size $|C|\ge p_t(n,k)-s(D)+1$ then $C$ solves any target $v\in\dD$ with cost at most $2^{\ecc_G(v)}$.

\begin{thm}
\label{t:main}
$\pi_t(P^{(k)}_n)=\ptnk$.
\end{thm}


As a result we obtain the following two corollaries.

\begin{cor}
\label{c:poly}
The $t$-fold pebbling number $\p_t(\Pnk)$ can be calculated in polynomial time (constant time if $k$ is known and linear time otherwise).
\end{cor}

The constant time follows from simply comparing $t(2^d-2)$ with $k(d-1)$, while the linear time follows from knowing that $k=\d(\Pnk)$ (which equals the degree of a simplicial vertex).

\begin{cor}
\label{c:expo}
Define the functions $M(n)=\lfloor (n-2)/(\lceil\lg n\rceil -2)\rfloor$ and $m(n)=\lfloor (n-2)/(\lceil\lg n\rceil -2)^2\rfloor$.
Then the pebbling exponents of paths are $\ep(P_2)=1$, $\ep(P_3)=\ep(P_4)=\ep(P_5)=2$, $\ep(P_6)=\ep(P_7)=\ep(P_8)=3$, 
and, for $n\ge 9$,
$M(n) - m(n)
\le \ep(P_n)\le M(n)$.
\end{cor}

We prove this in Section \ref{s:Correct}.
We also note that the upper bound is tight at $n=33$, $65$, $257$, and many other values, typically of the form $n=2^i+j$ for a few small values of $j$.
In fact, numerical evidence suggests that it may be tight at $n=2^i+1$ for most values of $i\ge 5$.
Moreover, except for values of $n$ in the range of something like $[2^i,2^i+i^2)$, it appears that the tighter bounds $M(n)-m(n)+2\le \ep(P_n)\le M(n)-m(n)+3$ may hold.

\section{Technical Lemmas} 
\label{s:tech}

\subsection{General Lemmas}
\label{ss:General}

Given a configuration $C$  of pebbles, a {\it potential move} is either a pair of pebbles sitting on the same vertex, or a single pebble sitting on a target vertex, which in either case is called {\it potential vertex}.
When counting the number of potential moves of $C$ in relation to a target $D$, we must be careful about counting too many singletons on a target vertex $v$: $\min\{C(v),D(v)\}$ of the pebbles there are singleton potentials, while the other $C(v)-D(v)$ pebbles must be counted in pairs (since they would need to move to solve other targets).
To say that $C$ has $j$ potential moves means that the $j$ pairs and singletons are pairwise disjoint.
For example, the configuration $C$ on 5 vertices $(v_1,\ldots,v_5)$ with values $(0,1,1,2,7)$ has 4 potential moves if the target $D$ has values $(2,0,0,0,0)$, and 5 potential moves if $D$ has values $(1,0,0,0,2)$.
The {\it potential} of $C$, $\pot(C)$, is the maximum $j$ for which $C$ has $j$ potential moves.
Because every solution that requires a pebbling move uses a potential move, the following fact is evident. 

\begin{fct} 
\label{f:pot}
Let $r$ be an empty vertex in a configuration $C$ with $\pot(C)<t$. 
Then $C$ is not $t$-fold $r$-solvable.
\end{fct}

Another useful tool is the following lemma.

\begin{lem}
\label{l:PotLem}
{\bf (Potential Lemma)}
Let $G$ be a graph on $n$ vertices. 
If $C$ is a configuration on $G$ of size $n+y$ ($y\ge 0$) having $z$ zeros, then $\pot(C)\ge\big\lceil \frac{y+z}{2}\big\rceil$.
\end{lem}



A {\it (induced) slide} from a potential vertex $v$ to a vertex $r$ is a (induced) path between $v$ and $r$ with a pebble on each interior vertex. 
Two slides are disjoint if the corresponding sets of  pebbles are disjoint.

Let $C$ be a configuration of $c$ pebbles $p_1,p_2,\ldots,p_c$ on a graph $G$. 
Let $\cS$ be a $t$-fold $r$-solution of $C$ moving pebbles $p_1,p_2,\ldots,p_t$ into the target vertex $r$. 
This means that $\cS$ is a sequence $\sigma_1, \sigma_2, \ldots, \sigma_h$ of pebbling steps after which a configuration with the pebbles  $p_1,p_2,\ldots,p_t$ on $r$ is obtained.
Assume that $\sigma_i$ moves the pebble $p_{\sigma_i }$ from vertex $v_{\sigma_i}$ to vertex $v'_{\sigma_i}$ and discards the pebble $p'_{\sigma_i}$.
We define the directed multigraph $G(\cS)$ to have the same vertex set as $G$, with a directed edge $(v_{\sigma_i},v'_{\sigma_i})$ for each pebbling step $\sigma_i$.
The following lemma of \cite{Moews} is very useful.
(It was given its descriptive name in \cite{BCCMW}.)

\begin{lem}
\label{l:NoCycle}
{\bf (No-Cycle Lemma)}
\cite{Moews}
If $C$ is $r$-solvable then there is an $r$-solution $\cS$ for which $G(\cS)$ is acyclic.
\end{lem}

Another simplifying concept may be assumed by the next lemma.
For $1\leq j \leq t$, let $\mfs_j$ be the subsequence of $\cS$ moving the pebble $p_j$ from its original position (say a vertex $v_j$) to the target $r$ (i.e. the subsequence 
of movements $\sigma_i$ such that $p_{\sigma_i }=p_j$)   , and let $\mfs$ be the subsequence of $\cS$ formed by the remaining pebbling steps.
We say that $\cS$ is a {\it tidy} solution if $\cS=\mfs, \mfs_1, \ldots, \mfs_t$, with each subsequence $\mfs_j$ consisting of moving the pebble $p_j$ along an \emph{induced} slide from the vertex $v_j$ to $r$.

We say that $\cS'$ is a {\it tidy rearrangement} of $\cS$ if $\cS'$ is a permutation of the pebbling steps of $\cS$ and $\cS'$ is tidy.
Since none of the pebbles $p_1,p_2,\ldots,p_t$ is discarded in the solution $\cS$, the pebbling steps of $\cS$ can be permuted to create a new $t$-fold $r$-solution sequence $\cS'$, namely, $\mfs$,$\mfs_1$,$\ldots$,$\mfs_t$.
Notice that when $\cS$ is {\it minimum} (i.e., using the fewest pebbling steps), the No-Cycle Lemma \ref{l:NoCycle} implies that each subsequence $\mfs_j$ consists of moving the pebble $p_j$ along an \emph{induced} slide from the vertex $v_j$ to $r$.
That is, $\cS'$ is tidy.
We record this as follows.

\begin{lem}
\label{l:Tidy}
{\bf (Tidy Lemma)}
If $\cS$ is a $t$-fold $r$-solution of a configuration $C$ on the graph $G$, then there is a tidy rearrangement $\cS'$ of $\cS$.
\end{lem}

For an $r$-solution $\s=\s_1, \ldots, \s_c$, we define its {\it cost} to equal $\cost(\s)=c+1$.
The idea is that we lose one pebble in each step, plus one more pebble that is placed on $r$ --- these pebbles cannot be used in subsequent $r$-solutions.
We say that $\s$ is {\it cheap} if $\cost(\s)\le 2^d$ (and {\it super-cheap} if $\cost(\s)< 2^d$), and define the parameter $q(G,r)$ to be the minimum number of pebbles $m$ so that every configuration of $m$ pebbles has a cheap $r$-solution.
Of course we always have $q(G,r)\ge \p(G,r)$.
For the particular graph $G=\Pnk$ with simplicial target $r$ we instead use the notation $q(n,k)$.

We say that a graph $G$ is $r$-(semi)greedy if every configuration of size at least $\pi(G,r)$
has a (semi)greedy $r$-solution; that is, every pebbling step in the solution decreases (does not increase) the distance of the moved pebble to $r$.

\begin{lem}
\label{l:CheapLemma}
{\bf (Cheap Lemma)}
\cite{AGHSemi2T}
Given the graph $G$ with target $r$ let $G^*$ be an $r$-greedy spanning subgraph of $G$ preserving distances to $r$.
Then any configuration of $G$ of size at least $\pi(G^*,r)$ is cheap; i.e. $q(G,r)\le \pi(G^*,r)$.
\end{lem}

In particular, if $T$ is a breadth-first-search spanning tree of $G$, rooted at $r$, then $q(G,r)\le \pi(T,r)$.
In our case we can choose $T$ to be a caterpillar with main path of length $d$, so that $\p(T,r)=2^d+n-d-1$; that is, $q(n,k)\le 2^d+n-d-1$.

Another useful lemma is the following.
The proof given in \cite{AGHSemi2T} for the case $t=1$ extends to all $t$.

\begin{lem}
\label{l:EdgeRemoval}
{\bf (Edge Removal Lemma)}
\cite{AGHSemi2T}
Given the graph $G$ with target $r$, if $e$ is an edge between two neighbors of $r$ then $\p_t(G,r)=\p_t(G-e,r)$.
\end{lem}

Let $S_i$ be the  slide from $v_n$ to $v_1$ with interior vertices  $v_{i_1}, v_{i_2}, \ldots , v_{i_{m_i}}$. 
Clearly, if $ j > j'$, then $i_j \leq i_{j'}$. 
If $S_i$ and $S_{\ell}$ are slides such that $i_1\leq \ell_1$, then we can assume that for every $j$, $i_j\leq \ell_j$.

Finally, we note a key property of $k$-connected graphs that follows from Menger's Theorem and Dirac's Fan Lemma (see Exercise 4.2.28 in \cite{West}.)

\begin{lem}
\label{l:menger}
Let $X$ and $Y$ be disjoint sets of vertices in a $k$-connected graph $G$.
For each $x\in X$ and $y\in Y$, let $u(x)$ and $w(y)$ be non-negative integers such that $\sum_{x\in X}u(x) = \sum_{y\in Y}w(y) = k$.
Then $G$ has $k$ pairwise internally disjoint $(X,Y)$-paths such that, for each $x\in X$ and $y\in Y$, $u(x)$ of them start at $x$ and $w(y)$ of them end at $y$.
\end{lem}

From Lemma \ref{l:menger} we obtain the following simple corollary.

\begin{cor}
\label{c:menger}
Let $X$ and $Y$ be disjoint sets of vertices in a $k$-connected graph $G$.
For each $x\in X$ and $y\in Y$, let $u(x)$ and $w(y)$ be non-negative integers such that $\sum_{x\in X}u(x)\ge j$ and $\sum_{y\in Y}w(y)\ge j$ for some $j\le k$.
Then $G$ has $j$ pairwise internally disjoint $(X,Y)$-paths such that, for each $x\in X$ and $y\in Y$, at most $u(x)$ of them start at $x$ and at most $w(y)$ of them end at $y$.
\end{cor}

\begin{proof}
Since $G$ is $k$-connected it is also $j$-connected.
Choose any set of $u'(x)\le u(x)$ and $w'(y)\le w(y)$ such that $\sum_{x\in X}u'(x) = \sum_{y\in Y}w'(y) = j$.
Then apply Lemma \ref{l:menger}.
\end{proof}

\subsection{Cutting Lemmas}
\label{ss:Cutting}

\begin{thm}
\label{t:internal}
Let $r$ be a vertex of a graph $G$ and denote by $E_r$ some set of edges between neighbors of $r$.
Suppose that $G-r-E_r$ has connected components $G_1,\ldots,G_j$ for some $j>1$, and define $G'_i$ to be the subgraph of $G$ induced by $V(G_i)\cup\{r\}$.
Then
\[\p_t(G,r)= \max_{\sum_{i=1}^jt_i=t+j-1} \sum_{i=1}^j \p_{t_i}(G'_i,r)-j+1.\]\
\end{thm}

\begin{proof}
We begin by letting $r$ and $E_r$ be as described in the hypothesis, namely, that $G-r-E_r$ has components $G_1,\ldots,G_j$, for some $j>1$.
By the Edge Removal Lemma \ref{l:EdgeRemoval}, $\p_t(G,r)=\p_t(G-E_r,r)$.

For any positive $t_1,\ldots,t_j$ such that $\sum_{i=1}^j t_i=t+j-1$ (i.e. $\sum_{i=1}^j (t_i-1)=t-1$), there exist $t_i$-fold $r$-unsolvable configurations $C_i$ on $G'_i$ of size $\p_{t_i}(G'_i,r)-1$.
The union $C=\cup_{i=1}^j C_i$ is therefore $t$-fold $r$-unsolvable, showing that $\p_t(G-E_r,r)\ge 1+\sum_{i=1}^j (\p_{t_i}(G'_i,r)-1) = \sum_{i=1}^j \p_{t_i}(G'_i,r)-j+1$.
Hence \[\p_t(G-E_r,r)\ge \max_{\sum_{i=1}^j t_i=t+j-1} \sum_{i=1}^j \p_{t_i}(G'_i,r)-j+1.\]

Let $C$ be a configuration of maximum size on $G$ that is not $t$-fold $r$-solvable, and let $C_i$ denote the restriction of $C$ to $G'_i$.
Without loss of generality $C(r)=0$, and so $|C|=\sum_{i=1}^j |C_i|$.
For each $i$, 
$C_i$ 
is $(t_i-1)$-fold but not $t_i$-fold $r$-solvable on $G'_i$; thus $|C_i|\le \p_{t_i}(G'_i,r)-1$ and $\sum_{i=1}^j (t_i-1)\le t-1$ (i.e. $\sum_{i=1}^j t_i\le t+j-1$).

Now the maximality of $|C|$ (and independence of $\{C_i\}$; that is, $G'_i$ and $G'_{i'}$ intersect only on $r$ and so no $C_i$ can affect another $C_{i'}$) implies that $\sum_{i=1}^j (t_i-1)=t-1$ and $|C_i| = \p_{t_i}(G'_i,r)-1$.
Hence $|C|=\sum_{i=1}^j (\p_{t_i}(G'_i,r)-1)$ and consequently \[\p_t(G-E_r,r)\le \max_{\sum_{i=1}^j t_i=t+j-1} \sum_{i=1}^j \p_{t_i}(G'_i,r)-j+1.\]\
This finishes the proof.
\end{proof}

Because every interior vertex of $\Pnk$ satisfies the hypothesis of Theorem \ref{t:internal}, we obtain the following corollary.

\begin{cor}
\label{c:internal}
Let $r=v_i$ be an interior vertex of $\Pnk$.
By removing the edges $v_hv_j$ between vertices of $N(r)$ 
with $h<i<j$,
we obtain two graphs $P_{n_1}^{(k)}$ and $P_{n_2}^{(k)}$, with $n_1+n_2=n+1$, that are joined at $r$, which is simplicial in both of them.
Then \[\p_t(\Pnk,r)= \max_{t_1+t_2=t+1} \p_{t_1}(P_{n_1}^{(k)},r)+\p_{t_2}(P_{n_2}^{(k)},r)-1.\]
\end{cor}

\subsection{Chordal Lemmas}
\label{ss:chordal}

\begin{lem}
\label{l:chordal} 
Let $r$, $v$ and $u$ be three vertices of a chordal graph $G$. 
If $v$ belongs to an induced path between $u$ and $r$, then any path between $u$ and $r$ contains either $v$ or a neighbor of $v$. 
Accordingly, $N[v]\setminus \{r,u\}$ separates $r$ from $u$.
\end{lem}

\begin{proof}
Since $v$ belongs to an induced path between $r$ and $u$, there is a minimal $r$-$u$-separator containing $v$. 
Since minimal separators in chordal graphs are cliques, every path between $r$ and $u$ contains either $v$ or a neighbor of $v$.
\end{proof}

\begin{lem}
\label{l:Greedy}
Chordal graphs are semi-greedy.
\end{lem}

\begin{proof}
Let $C$ be an $r$-solvable configuration on a chordal graph $G$.
Suppose that $\cS$ is a minimum $r$-solution and that $\cS$ is not semi-greedy.
Then, after a tidy rearrangement of $\cS$ if necessary (by the Tidy Lemma \ref{l:Tidy}), we can assume $\cS$ ends by moving a pebble $p$ along an \emph{induced} slide $Q:(u=v_1,v_2,\ldots, v_h=r)$ where  $\dist(v_i,r)< \dist(v_{i+1},r)$ for some $i\in\{1,\ldots,h-2\}$.
By Lemma \ref{l:chordal}, and the fact that $i+1<h-1$ (else $\dist(v_i,r)<\dist(v_{i+1},r)=1$, which would imply that $v_i=r$),
the path  $v_{i+1},v_{i+2}, \ldots, v_h=r$ has an interior vertex adjacent to $v_i$, in contradiction with the fact that $Q$ is an induced slide. 
\end{proof}

\begin{lem}
\label{l:C'} 
Let $v$ be a potential vertex of a  configuration $C$ on a graph $G$. 
Let $u$ be any other vertex  which is separated from $r$ by $N[v]\setminus \{r,u\}$. 
Let $C'$  be the same configuration as $C$ except that the potential vertex $v$ is changed to $u$; that is, $C'(v)=C(v)-2$, $C'(u)=C(u)+2$, and $C'(x)=C(x)$ otherwise. 
If $C$ is $t$-fold $r$-unsolvable, then $C'$ is $t$-fold $r$-unsolvable.
\end{lem}

\begin{proof} 
Suppose otherwise that $C'$ is $t$-fold $r$-solvable and let $\cS$ be a minimum $t$-fold $r$-solution; by Lemma \ref{l:Greedy}, $\cS$ is semi-greedy.
We can assume that $\cS$ takes a pebble $p$ from $u$ to $r$.
Therefore, by a tidy rearrangement if necessary (by the Tidy Lemma \ref{l:Tidy}), we can also  assume that $\cS$ ends with a subsequence $\mfs_p$ moving the pebble $p$ along an slide $Q$ between $u$ and $r$. 
Since $N[v]\setminus \{r,u\}$ separates $u$ from $r$, there exists an interior vertex of $Q$, say $x$, that belongs to $N[v]$. 
This implies that replacing in $\cS$ the subsequence  $\mfs_p$ by the movement of the pebble $p$ from $v$ to $r$ along the part of the slide $Q$ between $x$ and $r$, we obtain a semi-greedy $t$-fold $r$-solution of the original configuration $C$, a contradiction.
\end{proof}

The following theorem is not used in this paper, but is likely to be useful in future work.

\begin{thm}
\label{t:Simplicial}
If $G$ is chordal and $C$ is $t$-fold $r$-unsolvable then there is a $t$-fold $r$-unsolvable configuration $C'$ with $|C'|=|C|$ and every potential vertex is simplicial.
\end{thm}

\begin{proof}
Suppose that $C$ is $t$-fold $r$-unsolvable and has a non-simplicial potential vertex $v$. 
If $v$ is adjacent to $r$, the proof is trivial. 
Otherwise, there exists a simplicial vertex $u$ such that $N[v]\setminus \{r,u\} $ separates $u$ from $r$. 
Therefore, the proof follows  by Lemma \ref{l:C'}.
\end{proof}

\subsection{Pebbling Number for Trees}
\label{ss:trees}

We will make use of the following well known theorem of Chung \cite{Chung} on the $t$-fold pebbling number of a tree with target vertex $r$, using the notion of a {\it maximum $r$-path partition} $\cP$.
One can compute such a $\cP$ iteratively as follows.
Beginning with $H=T$, $W=\{r\}$, and $\cP=\mt$, we choose a longest path $P$ in $H$ having one endpoint in $W$. 
Then we add $P$ to $\cP$, add its vertices to $W$, remove its edges from $H$, and repeat.
The construction yields $\cP=\{P_1,\ldots,P_\ell\}$, with subscripts signifying the order of inclusion in $\cP$, and with $\len(P_i)\ge\len(P_{i+1})$.
Note that the number of leaves of $T$ equals $\ell+1$.
In general, we write $\cP(T,r)$ for the maximum $r$-path partition of $T$.

\begin{thm}
\cite{Chung}
\label{t:TreeFormula}
Let $T$ be a tree with maximum $r$-path partition $\cP=\{P_1,\ldots,P_\ell\}$, with each $P_i$ having length $a_i$.
Then $\p_t(T,r) = (t2^{a_1}-1) + \sum^\ell_{i=2}(2^{a_i}-1) + 1 = t2^{a_1} +\sum^\ell_{i=2} 2^{a_i}-\ell+1$.
\end{thm}

The pebbling number $\p_t(T)$ is given by choosing $r$ to be a leaf of a longest path of $T$.
We write $\cP(T)$ for the maximum path partition of $T$, equal to $\cP(T,r)$ for this choice of $r$.

\subsection{Path Power Lemma}
\label{ss:pathpower}

\begin{lem}
\label{l:interior}
For $1\leq i \leq 2$, let  $n_i\geq 1$ and  $t_i\geq 1$. If $n_1+n_2=n+1$ and  $t_1+t_2=t+1$, then 
 $p_{t_1}(n_1,k)+p_{t_2}(n_2,k)-1\le p_{t}(n,k)$.
\end{lem}

\begin{proof} Let $d_i$ be the diameter of $P_{n_i}^k$, and $d$ be 
the diameter of $P_n^k$. Observe that $d\le d_1+d_2\le d+1$.
Consider the following three cases.
\begin{enumerate}[listparindent=1.5em]
    \item 
    $t_1\ge t_0(n_1,d_1)$ and $t_2\ge t_0(n_2,d_2)$.

    Suppose, without loss of generality, that $d_1\ge d_2$. If either $d\leq 2$ or $d_1=d$ (which implies $d_2=1$),
    the proof is simple, so assume $d\geq 3$ and $d_1<d$. 
    Because $d_1+d_2\ge d$ we have
    \begin{align*}
        p_{t_1}(n_1,k)+p_{t_2}(n_2,k)-1
        &= l_{t_1}(n_1,k)+l_{t_2}(n_2,k)-1\\
        &= t_12^{d_1}+n_1-2-k(d_1-1)+t_22^{d_2}+n_2-2-k(d_2-1)-1\\
        &\le (t_12^{d_1}+t_22^{d_2})+(n-4)-k(d-2)\\
        &= (t_12^{d_1}+t_22^{d_2})+(n-4)-k(d-1)+k\\
        &\le (t2^{d_1}+2^{d_2})+(n-4)-k(d-1)+k \\
        &\le (t2^{d-1}+2^1)+(n-4)-k(d-1)+k\,\,\, \hbox{ since } d_1<d\\
        &= (t2^{d-1}+k)+(n-2)-k(d-1).
    \end{align*} 
    Note that since $d\ge 3$ we have $d_1\ge d/2 > 1$. 
    Thus, because $t_1\ge t_0(n_1,d_1)$, we know that $k\le t_1(2^{d_1}-2)/(d_1-1)\le t_1(2^{d_1}-2)\le t2^{d-1}$, and so the above amount is at most $t2^d+(n-2)-k(d-1) = \ltnk\le \ptnk$.
    \item 
    $t_1\le t_0(n_1,d_1)$ and $t_2\le t_0(n_2,d_2)$.

    Here we have
    \begin{align*}
        p_{t_1}(n_1,k)+p_{t_2}(n_2,k)-1
        &= w_{t_1}(n_1)+w_{t_2}(n_2)-1\\
        &= 2t_1+n_1-2+2t_2+n_2-2-1\\
        &= 2(t+1)+(n+1)-5\\
        &= 2t+n-2\\
        &= \wtn\\
        &\le p_t(n,k).
    \end{align*}
    \item 
    $t_1> t_0(n_1,d_1)$ and $t_2< t_0(n_2,d_2)$, or $t_1< t_0(n_1,d_1)$ and $t_2> t_0(n_2,d_2)$.

    Without loss of generality, we may assume for former, in which case we have $t_1(2^{d_1}-2)> k(d_1-1)$ and $t_2(2^{d_2}-2)< k(d_2-1)$.
    In particular, each $d_i\ge 2$, and so we will make use of the fact that $\left(\frac{2^{j}-2}{j-1}\right)$ is an increasing function for $j\ge 2$.
    Since $t_1> t_0(n_1,d_1)$, this implies that 
    \begin{equation}
        k< t_1\left(\frac{2^{d_1}-2}{d_1-1}\right) < 
        t\left(\frac{2^d-2}{d-1}\right),
    \label{e:Long}
    \end{equation}
    which means that $\ptnk=\ltnk=t2^d+n-2-k(d-1)$.
    Now we compute
    \begin{align*}
    \ptnk - \left[p_{t_1}(n_1,k)+p_{t_2}(n_2,k)-1\right] - 1
        &= \left[t2^d+(n-2)-k(d-1)\right]\\
        &\ \ - \left[t_12^{d_1}+(n_1-2)-k(d_1-1)+2t_2+(n_2-2)\right] - 1\\
        &= \left[t2^d-k(d-1)\right] - \left[t_12^{d_1}-k(d_1-1)+2t_2\right]\\
        &\ge \left[t(2^d-2)-k(d-1)\right] - \left[t_1(2^{d_1}-2)-k(d_1-1)\right],
    \end{align*}
    which we will show is positive.
    Indeed, because of the second inequality in (\ref{e:Long}) we know that
    \[t\left(\frac{2^d-2}{d-1}\right)\ge t_1\left(\frac{2^{d_1}-2}{d_1-1}\right),\]
    so that
    \[t(2^d-2)-k(d-1)\ge t_1(2^{d_1}-2)\left(\frac{d-1}{d_1-1}\right)-k(d-1),\]
    which we will show is greater than $t_1(2^{d_1}-2)-k(d_1-1)$ as follows, using the first inequality of (\ref{e:Long}) in the final step.
    \[\left[t_1(2^{d_1}-2)\left(\frac{d-1}{d_1-1}\right)-k(d-1)\right] - \left[t_1(2^{d_1}-2)-k(d_1-1)\right]\qquad\qquad\qquad\qquad\qquad\]
    \vspace{-0.2in}
    \begin{align*}
        \qquad\qquad\qquad\qquad\qquad\qquad
        &= t_1(2^{d_1}-2)\left[\frac{d-1}{d_1-1}-1\right]-k(d-d_1)\\
        &= \left[t_1\left(\frac{2^{d_1}-2}{d_1-1}\right)-k\right](d-d_1)\\
        &= \left[t_1(2^{d_1}-2)-k(d_1-1)\right] \left(\frac{d-d_1}{d_1-1}\right)\\
        & >0.
    \end{align*}
\end{enumerate}
This completes the proof.
\end{proof}

\section{Proof of Theorem \ref{t:tree}}
\label{s:trees}

In this section we verify the Strong $t$-Target Conjecture \ref{j:tTargetStrong} for trees.
For a $(C,D)$-solution $\s$, define $C[\s]$ to be the subconfiguration of $C$ consisting of only those pebbles used by $\s$, and, for $v\in\dD$, define $\s[v]$ to be the pebbling moves of $\s$ consisting of only those pebbles used to solve $v$.

\begin{proof}
The result is trivially true when $s(D)=1$, which means it is true also for $t=1$ and for $|V(T)|=1$, so we assume that $s(D)\ge 2$.
We now proceed by induction on any of these parameters.
Throughout, we will use the fact that every tree $T$ satisfies $\p_t(T)\ge\p_{t-1}(T)+2$; this is evident from Theorem \ref{t:TreeFormula} because $2^{a_1}\ge 2$.

Let $|C|=\p_t(T)-s(D)+1$.
Then 
$|C|= \p(T)+(t-1)2^d-s(D)+1\ge \p(T)+(t-1)(2^d-1)\ge \p(T)$.
If some $r\in\dC\cap\dD$, then $s(D-r)\ge s(D)-1$ and so
\begin{align*}
    |C-r|
    &=|C|-1\\
    &= (\p_t(T)-s(D)+1)-1\\
     &\ge \p_t(T)-s(D-r)-1\\
    &\ge \p_{t-1}(T)-s(D-r)+1.
\end{align*}
If some $r\in\dD$ has $\ecc(r)<d$, then induction on $t$ implies that $C$ has an $r$-solution $\s$ with cost at most $2^{\ecc(r)}\le 2^d-1$.
Then $s(D-r)\ge s(D)-1$ and so
\begin{align*}
    |C-C[\s]|
    &=|C|-\cost(\s)\\
    &\ge (\p_t(T)-s(D)+1)-(2^d-1)\\
    &\ge \p_{t-1}(T)-s(D-r)+1.
\end{align*}
Similarly, if some $r^2\in\dD$ (i.e. $D(r)\ge 2$) has $\ecc(r)=d$, then induction implies that $C$ has an $r$-solution $\s$ with cost at most $2^d$.
Then $s(D-r)=s(D)$ and so
\begin{align*}
    |C-C[\s]|
    &=|C|-\cost(\s)\\
    &\ge (\p_t(T)-s(D)+1)-2^d\\
    &= \p_{t-1}(T)-s(D-r)+1.
\end{align*}
In all three cases, induction on $t$ implies that $C-C[\s]$ is $(D-r)$-solvable, making $C$ $r$-solvable, and any target can be solved with cost at most  $2^{\ecc_T(v)}$.

Therefore we may assume that $\dC\cap\dD=\mt$ and that if $r\in\dD$ then $r$ is a singleton leaf with $\ecc(r)=d$.
Note that $\p_t(T)\ge\p_t(T-v)+1$ for any leaf $v$, because each $a_i\ge 1$.
Since $t\ge 2$ we may choose distinct targets $r$ and $v$ in $\dD$.
Then
\begin{align*}
    |C|
    &= \p_t(T)-s(D)+1\\
    &\ge (\p_{t-1}(T)-s(D)+1)+1\\
    &= \p_{t-1}(T)-s(D-v)+1,
\end{align*}
and so, by induction on $t$, there is an $r$-solution $\s$ of cost at most $2^d$.
Therefore
\begin{align*}
    |C-C[\s]|
    &\ge (\p_t(T)-s(D)+1)-2^d\\
    &= \p_{t-1}(T)-s(D)+1\\
    &= \p_{t-1}(T)-s(D-r)\\
    &\ge \p_{t-1}(T-r)-s(D-r)+1,
\end{align*}
which implies, by induction on $t$ and $|V(T)|$, that $C-C[\s]$ is $(D-r)$-solvable and that any $u\in\dD-r$ has a solution of cost at most $2^\ecc(u)$.
Combined with the prior $r$-solution, this proves the theorem.
\end{proof}

\section{Proof of Theorem \ref{t:Dsolve}}
\label{s:proof_D_solve}

In this section we verify the Strong $t$-Target Conjecture \ref{j:tTargetStrong} for powers of paths.
We begin by introducing some important notation.

For pebbling functions $F_1$ and $F_2$, define the function $F_1\pm F_2$ by $(F_1\pm F_2)(v)=F_1(v)\pm F_2(v)$, as well as the functions $F_1\wedge F_2$ and $F_1\vee F_2$ by $(F_1\wedge F_2)(v)=\min\{F_1(v),F_2(v)\}$ and $(F_1\vee F_2)(v)=\max\{F_1(v),F_2(v)\}$.
(Thus, if $F=F_1\wedge F_2$ we have $\dF=\dF_1\cap\dF_2$, while if $F=F_1\vee F_2$ we have $\dF=\dF_1\cup\dF_2$.)

Let $F$ be a pebbling function on $G=\Pnk$.
Denote by $\Gij$ the subgraph of $G$ induced by the vertices $\{v_i,v_{i+1},...,v_j\}$, and define $\Fij$ to be the restriction of $F$ to $\Gij$.
Define the {\it pebbling arrangement} of $F$, $\cA(F)=\langle a_1,\ldots,a_m\rangle $, where $m=n+|F|$, to be the sequence 
\[\langle v_1, v_{1,1}, \ldots, v_{1,{F(v_1)}}, v_{2}, v_{2,1}, \ldots, v_{2,{F(v_2)}}, \ldots, v_{n}, v_{n,1}, \ldots, v_{n,{F(v_n)}}\rangle.\]
Also, define the {\it labeling} of $\dF$ to be the sequence \[\cL(F)=(v_{1,1}, \ldots, v_{1,{F(v_1)}}, v_{2,1}, \ldots, v_{2,{F(v_2)}}, \ldots, v_{n,1}, \ldots, v_{n,{F(v_n)}}).\]
For $1\le i\le n$, we write $V_i=\{v_i, v_{i,1}, \ldots, v_{i,F(v_i)}\}$.
Given a pebbling arrangement $\cA=\langle a_1, \ldots. a_m\rangle$, define $\Aaij$ to be $\langle a_i, \ldots. a_j\rangle$ if $a_i=v_h$, for some $h$, and $\langle v_h, a_i, \ldots, a_j\rangle$ if $a_i\in V_h-\{v_h\}$, for some $h$.
Define $\Faij$ to be the pebbling function whose pebbling arrangement $\cA(\Faij)$ equals $\Aaij$.

Given a size $t$ distribution $D$ with $s(D)<n$, let $j$ be the least index with $D(v_j)=0$ and define the configuration $W_D$ by $W_D(v)=0$ for all $v\in\dD$, $W_D(v_j)=2t-1$, and $W_D(v)=1$ otherwise.
Because $\pot(W_D)=t-1$, $W_D$ is $D$-unsolvable.
Notice also that $|W_D|=n-s(D)+2t-2=(\ptnk-s(D)+1)-1$.

In general, call a configuration $C$ $D$-{\it small} if $\pot(C)=|D|-1$, and a graph $G$ $D$-{\it small} if some $D$-unsolvable configuration of maximum size is $D$-small.
We say that $G$ is $t$-{\it small} if it is $D$-small for all $|D|=t$.
In addition, define a maximum-size $D$-small configuration $C$ to be {\it canonical} whenever the following holds:
if $s(D)<n$ then $C(v)=0$ for all $v\in\dD$ and $C(v)$ is odd for all $v\not\in\dD$;
if $s(D)=n$ then, for some $v$ with $D(v)=a=\min_u D(u)$, $C(v)=2t-a-1$ and $C(u)=0$ for all $u\not= v$.
A configuration $C$ is called {\it stacked} if there is some vertex $v$ such that $C(v)>0$ and $C(u)=0$ for all $u\not= v$. 
Observe that the canonical configuration in the $s(D)=n$ case above is stacked.
In \cite{CCFHPST}, the following theorem was proven by using the appropriate stacked configuration.

\begin{thm}
\label{t:piKn}
If $s(D)=n$ then $\pi(K_n,D)=2|D|-\min_u D(u)$.
\end{thm}

This yields the following corollary.

\begin{cor}
\label{c:canonKn}
If $s(D)=n$ then $K_n$ has a canonical $D$-small configuration.
\end{cor}

Theorem \ref{t:piKn} is part of the following, more general ``Stacking Theorem'' proved in \cite{Sjost} and \cite{VuoWyc}.

\begin{thm}
\label{t:Stack}
For every graph $G$, if $s(D)=n(G)$ then there exists a $D$-unsolvable configuration of maximum size that is stacked.
\end{thm}

\begin{lem}
\label{l:canonical}
If $\p(\Pnk,D)\le n+2|D|-2-s(D)+1$ and if $s(D)<n$ then $\Pnk$ has a canonical $D$-small configuration.
\end{lem}

\begin{proof}
Suppose that $s(D)<n$ and let $C$ be a configuration of size $n+2|D|-2-s(D)$ such that $C(v)=0$ for all $v\in\dD$ and $C(v)$ odd for all $v\not\in\dD$.
Such a configuration exists because, for $t=|D|$, we can place $(t-1)$ pairs of pebbles the vertices not in $\dD$, and then one additional pebble on each such vertex, which amounts to $2(t-1)+(n-s(D))=|C|$ pebbles in total.
Now, such a $C$ then clearly has $\pot(C)=\pot(C')=t-1$; that is, $C$ is $D$-small.
Moreover, this implies that $\p(\Pnk,D)\ge |C|+1 = n+2|D|-2-s(D)+1$, and so $|C|=\pi(\Pnk,D)-1$; that is, $C$ is maximal $D$-unsolvable, and hence canonical.
\end{proof}

\begin{fct}
\label{f:WideSupport}
Suppose that $\Pnk$ is $t$-wide, and let $d=\diam(\Pnk)$ and $D$ be a distribution of size $t$.
If $d\ge 2$ then $s(D)<n$.
Contrapositively, if $s(D)=n$ then $d=1$; i.e. $\Pnk=K_n$.
\end{fct}

\begin{proof}
This follows immediately from the inequality $s(D)\le t\le k(d-1)/(2^d-2)\le k/2<n$.
\end{proof}

We prove Theorem \ref{t:Dsolve} by proving the following stronger result.

\begin{thm}
\label{t:DsolveStrong} 
Let $k$, $n$ and $t$ be positive integers. 
Assume that $D$ is a target distribution on $G=\Pnk$ of size $t$.
Then $\p(G,D)\le\ptnk-s(D)+1$.
In addition, if $G$ is $t$-wide then it is $D$-small, while if $G$ is strictly $t$-long then any configuration $C$ on $G$ of size $|C|\ge p_t(n,k)-s(D)+1$ solves any target $v\in\dD$ with cost at most $2^{\ecc_G(v)}$.
\end{thm}

\begin{proof} For given $k$, $n$ and  $t$, let  $D$ be a target distribution on $P_n^k$ with $t=|D|$ and $s=s(D)$.
Observe that 
the $D$-unsolvable configuration $W_D$ witnesses that $\p(\Pnk,D)\ge \ptnk-s(D)+1$.
Hence, when $\Pnk$ is $t$-wide and $s(D)<n$,  showing that $\p(\Pnk,D)\le \ptnk-s(D)+1$ will also imply that $\p(\Pnk,D)= \ptnk-s(D)+1=|W_D|+1$, and so the second statement of the theorem follows immediately. Anyway, in the arguments that follow,
sometime we  first prove the second statement and then we use it to prove the upper bound.  

We start using induction on  $k$.
The case $k=1$ is a path, which has been proven in Theorem \ref{t:tree} (the only $t$-wide path is the 1-wide $P_2$), so we may assume that $k>1$. 

Now we proceed by induction on $n$.

\subsection{Base Step: $n\le k+1$}
\label{ss:Base}

Here $\Pnk=K_n$ and $d=1$, so $\ptnk=2t+n-2$.
Let $C$ be a pebbling configuration with $|C|\ge 2t+n-2-s(D)+1$ and let $H=C\wedge D$, with $|H|=h$.
Define $C'=C-H$ and $D'=D-H$, so that $|C'|=|C|-h$ and $|D'|=|D|-h$.
Observe that $C'$ has at least $s(D')$ zeros, so that, by the Potential Lemma \ref{l:PotLem},
\begin{align*}
    \pot(C')
    &\ge \bigg\lceil\frac{(2t-s(D)-1-h)+s(D')}{2}\bigg\rceil\\
    &= (t-h)+\bigg\lceil\frac{h-1+s(D')-s(D)}{2}\bigg\rceil\\
    &\ge (t-h),\\
\end{align*}
since $s(D')\ge s(D)-h$.
Hence we can solve the $h$ targets of $H$ identically and then solve the remaining $t-h$ targets of $D-H$ with the $t-h$ potential of $C-H$.
This completes the upper bound for the case $n\le k+1$.
As a consequence, $\Pnk$ is $D$-small in this case.
In addition, every minimal solution has cost at most $2=2^d$.

\subsection{Induction Step: $n\ge k+2$}
Here we have $d\ge 2$. We use induction on $t$.

\subsubsection{Base Case: $t=1$}
\label{sss:tbase}

Here $G$ can be $t$-long, but never strictly so; thus we will not need to calculate the cost of a solution.

In this case we have $s=1$.
Thus, if $d=2$  then the result holds by Theorem 3 of $\cite{AGHkConnU}$ because $G$ has a universal vertex.
So we will assume that $d\ge 3$.
In addition, if $k= 2$, then the result holds by Theorem 3.3 of \cite{AGH2P}, so we will assume that $k\ge 3$.

Let $r$ be the target vertex and suppose that $r$ is internal.
Then $\Pnk-r$ consists of two components $G_1$ and $G_2$ such that, for each $i$, the subgraph of $\Pnk$ induced by $V(G_i)\cup\{r\}$ is isomorphic to $P_{n_i}^{(k)}$, with $n_1+n_2=n+1$.
Now Corollary \ref{c:internal} states that (since the only solution to $t_1+t_2=2$ is $t_1=t_2=1$)
\[\p(\Pnk,D)=\p_1(\Pnk,r)= \p_1(P_{n_1}^{(k)},r)+\p_1(P_{n_2}^{(k)},r)-1,\]
which we can write as
\[\p(\Pnk,D)\le p_1(n_1,k)+p_1(n_2,k)-1,\]
by using induction on $n$.
Then Lemma \ref{l:interior} reveals that this is at most $p_1(n,k)$ , as required.

Thus we may assume that $r$ is simplicial; i.e. that $D(v_1)=1$. 
Let $C$ be of size $p_1(n,k)-1+1$; clearly we can assume 
$C(v_1)=0$.
Hence there is a big vertex $u$, which means that there is an empty $(v_1,u)$-separator; i.e, 
there are $k$ consecutive empty vertices $v_{i+1}$, $\cdots $, $v_{i+k}$, for some $i\ge 1$.
Recall that $t_0=t_0(k,d)=k(d-1)/(2^d-2)$.

\begin{enumerate}[listparindent=1.5em]

    \item 
    \underline{\bf Near Case: $i=1$.}
    \bigskip\\
    Here we have $C(v_2)=\cdots=C(v_{k+1})=0$. 
    To put a pebble on $v_1$, it is enough to put two pebbles on $v_{k+1}$. 
    Let $G'=G_{[k+1,n]}$, with configuration $C'=C_{[k+1,n]}$. 
    Notice that $G'=P_{n'}^{(k)}$, where $n'=n-k$ and $d'=d-1$, and that $|C'|=|C|$.

    \begin{enumerate}[listparindent=1.5em]

        \item 
        \underline{\bf Subcase: $1\le t_0(k,d)$.}
        \bigskip\\
        In this case we have $|C| = w_1(n) = n$.
        If $2\le t_0(k,d-1)$ then, by induction on $n$, we have $\pi_2(G') = w_2(n') = n-k+2 < n = |C'|$, and so we can place two pebbles on $v_{k+1}$.

        On the other hand, if $2>t_0(k,d-1)$ then $2(2^{d-1}-2)>k(d-2)$ and so, by induction on $n$, we have $\pi_2(G') =  2(2^{d-1})+(n-k)-k(d-2)-2 = n+[(2^d-2)-k(d-1)]\leq n = |C'|$; thus we can place two pebbles on $v_{k+1}$. 
        
        We can also conclude that in this case $G$ is $D$-small.
        
        \item 
        \underline{\bf Subcase: $1>t_0(k,d)$.}
        \bigskip\\
        In this case we have $|C| = l_1(n,k) = 2^d+n-2-k(d-1)$.
        Because $k\ge 3$, we know that $k\ge 2$.
        Thus $-k(d-2)\ge 2-k(d-1)$, which implies that $2(2^{d-1}-2)-k(d-2)\ge (2^d-2)-k(d-1)>0$.
        Hence $2>t_0(k,d-1)$.

        By induction on $n$ we have $\p_2(G') = l_2(n',k) = 2(2^{d-1})+(n-k)-2-k(d-2) = |C'|$, and so we can place two pebbles on $v_{k+1}$.

    \end{enumerate}

    \item 
    \underline{\bf Far Case: $i\ge 2$.}
    \bigskip\\
        
    Now we have that the first empty cutset is $v_{i+1}$, $\cdots $, $v_{i+k}$, for some $i>1$.
    Hence $C(v_i)>0$ and there is a slide from $v_i$ to $v_1$, so we may assume that $C(v_i)=1$. 
    Therefore, to put a pebble on $v_1$, it is sufficient to put one more pebble on $v_i$ from the configuration $C'=C_{[i,n]}-v_i=C_{[i+1,n]}$ on $G'=G_{[i,n]}$.
    Notice that $G'=P_{n'}^{(k)}$, where $n'=n-i+1$, and that $|C'|\geq |C|-i+1$.

    The diameters of $G$ and $G'$ are related by  \[k(d'-1)+b'=n'-2=n-i+1-2=k(d-1)+b-i+1,\]
    so that $k(d-d')=i+(b'-b)-1$.

\begin{enumerate}[listparindent=1.5em]

        \item 
        \underline{\bf Subcase: $1\le t_0(k,d)$.}
        \bigskip\\
        In this case we have $|C| = w_1(n) = n$.
        Because, for $d\ge 2$ and fixed $k$, $t_0(k,d)$ is a strictly decreasing function of $d$, and $d'\le d$, we have $1\le t_0(k,d')$.
        Then, by induction on $n$, we have $\pi_1(G') = w_1(n') = n-i+1 \le |C'|$, and so we can place a pebble on $v_i$.\\
        Again, we can conclude that in this case $G$ is $D$-small.

        \item 
        \underline{\bf Subcase: $1>t_0(k,d)$.}
        \bigskip\\
        In this case we have $|C| = l_1(n,k) = 2^d+n-2-k(d-1)$ and $(2^d-2)>k(d-1)$.
        If $1\le t_0(k,d')$ then, by induction on $n$, we have $\pi_2(G') = w_2(n') = n-i+1 < [2^d+n-2-k(d-1)]-i+1 \le |C'|$, and so we can place two pebbles on $v_{k+1}$.

        On the other hand, if $1>t_0(k,d')$ then $(2^{d'}-2)>k(d'-1)$.
        Because $t_0(k,d)$ is decreasing in $d$, and $d'\le d$, we have 
        \begin{align*}
            \left(\frac{2^{d'}-2}{d'-1}\right)-k 
            &\le  \left(\frac{2^d-2}{d-1}\right)-k\\
            &\le  \left(\frac{2^d-2}{d'-1}\right)-k\left(\frac{d-1}{d'-1}\right),
        \end{align*}
        which implies that $(2^{d'}-2)-k(d'-1) \le (2^d-2)-k(d-1)$.
        Hence, by induction on $n$, we have $\pi_1(G') = 2^{d'}+(n-i+1)-k(d'-1)-2
        \leq [2^d+n-2-k(d-1)]-i+1 \le |C'|$, and so we can place two pebbles on $v_{k+1}$, each at cost at most $2^{\ecc_{G'}(v_{k+1})}$, and therefore at total cost at most $2^{\ecc_{G}(v_{k+1})}$.
    \end{enumerate}
\end{enumerate}

This completes the 
proof for the case $t=1$.

\subsubsection{Inductive Case: $t\ge 2$.}

Let $C$ be of size $p_t(n,k)-s(D)+1$. We will prove $C$ solves $D$.
If $C(x)\wedge D(x)>0$ for some $x$, then we use a pebble on $x$ to solve a target on $x$, and then use $C-x$ to solve $D-x$ by induction on $t$. 
This is possible because $\ptnk> p_{t-1}(n,k)+2$ for all $d\ge 2$ and all $t$, and because $s(D-x)\ge s(D)-1$, so that
\begin{align*}
    |C-x|
    &= (\ptnk-s(D)+1)-1\\
    &\ge p_{t-1}(n,k)-s(D-x)+1\ .
\end{align*} 
Hence we may assume that $\dC\cap\dD=\mt$ (which in turn  implies that $s(D)<n$), and no vertex adjacent to a target vertex is big.

If $C(v_1)=D(v_1)=0$, then we use induction on $n$, since \begin{align*}
    |C_{[2,n]}|
    &= \ptnk-s(D)+1\\
    &\ge [p_t(n-1,k)-s(D_{[2,n]})+1]+1\ .
\end{align*}
The same can be said by symmetry for $v_n$.
Hence we may assume that each of $v_1$ and $v_n$ have either a pebble or a target on it.

We consider three cases regarding the size of $t$.
Recall that $t_0=t_0(n,k)=k(d-1)/(2^d-2)$.

\begin{enumerate}[listparindent=1.5em]
        
    \item 
    \label{tWide}
    \underline{\bf Wide Case: $t\le t_0$.}
    \bigskip\\
    We consider two cases regarding the size of $d$.
            
    \begin{enumerate}[listparindent=1.5em]

        \item 
        \underline{\bf Subcase: $d=2$.}
        \bigskip\\
        Here we have $t\le k(d-1)/(2^d-2)=k/2$.
        Let $n=k+2+b$ and $|C|=2t+n-2-s+1$, where $s=s(D)$.
        Define $M=\{v_{b+2},\ldots,v_{k+1}\}$ and note that every vertex in $M$ is dominating.
        This means that every pair of potentials can solve any target via pebbling steps through any vertex in $M$.
        Because $\dC\cap\dD=\mt$, we first observe that $z\ge s$, where $z$ is the number of zeros of $C$, which, by Lemma 10, implies that
        \[\pot(C)\ge \left\lceil\frac{(2t-2-s+1)+s}{2}\right\rceil= t\ .\]
        Since $t\ge 2$ and any pair of potentials solves any root $r$ through any vertex of $M$, we have a solution  at cost $4=2^{\ecc(r)}$, 
        because any root has eccentricity 2.

        Next, for each $x\in\dD$, define $u(x)=2D(x)$ and, for each potential vertex $y$ of $C$, define $w(y)=2\lfloor C(y)/2\rfloor$ (i.e. twice the number of potentials at $y$).
        Then, since $t\le k/2$, Corollary \ref{c:menger} (with $j=2t\leq k$) yields $2t$ internally disjoint $(D,B)$-paths $\cP$, where $B$ is the set of potentials of $C$.
        Let $\r$ be the number of targets in $\dD$ that have a path in $\cP$ with no zeros on it.
        Then there are $t-\r$ targets in $\dD$, all of whose paths in $\cP$ contain at least one zero; that is, there are at least $2(t-\r)$ zeros of $C$ different from $\supp(D)$.
        Hence
        \begin{align*}
            \pot(C)
            &\ge \left\lceil\frac{(2t-2-s+1)+(s+2t-2\r)}{2}\right\rceil\\
            &= \left\lceil\frac{4t-1-2\r}{2}\right\rceil\\
            &
            =2t-\r\\
            &= \r+2(t-\r)\ .
        \end{align*}
        Thus we can solve the remaining $t-\r$ targets via pairs of potentials, which solves $D$.
        This completes the upper bound for the subcase $d=2$ of the case $t\le t_0$.
        As a consequence, $\Pnk$ is $D$-small in this subcase.\\

        \item 
        \underline{\bf Subcase: $d\ge 3$.}
        Notice that, if $t'\le t$ and $n'\le n$, then $P_{n'}^{(k)}$ is $t'$-wide: $d'=\diam(P_{n'}^{(k)})\le\diam(\Pnk)=d$, and so
        \[t'\le t\le k\left(\frac{d-1}{2^d-2}\right)\le k\left(\frac{d'-1}{2^{d'}-2}\right),\]
        since $d\ge 3$ and the function $(x-1)/(2^x-2)$ is decreasing for $x>1$.        
        Because $P_{n'}^{(k)}$ is $t'$-wide, we may assume that $P_{n'}^{(k)}$ is $t'$-small for all such $n'$ and $t'$ (except for when $n'=n$ and $t'=t$); that is, given any target configuration $D'$ of size $t'$, among the $D'$-unsolvable configurations of maximum size there exist at least one with $t'-1$ potential movements.
        
        Before proving that $C$ solves $D$, we will prove that $\Pnk$ is $D$-small; i.e. we will prove that there exists a pebbling configuration $C^*$ of maximum size among the $D$-unsolvable configurations s.t. $\pot(C^*)=|D|-1=t-1$.

        Write $D=\{v_{i_1},\ldots,v_{i_t}\}$, with $i_j\le i_{j+1}$ for all $1\le j<t$, and let $C'$ be $D$-unsolvable of maximum size.
        Moreover, among all such configurations, we may choose $C'$ to have the fewest pebbles in $G_{[1,k]}$; i.e. $c=|C'_{[1,k]}|$ is minimum.
        Also, since $C'$ is $D$-unsolvable, we have $|C'|\geq
        |W_D|\geq w_t(n)-s(D)+1$, and so,  from above, we can assume that $\dC'\cap\dD=\mt$, that no big vertex of $C'$ is adjacent to a vertex of $D$, and that each of $v_1$ and $v_n$ have either a target or a pebble.

        Observe that $C'$ is $(D-v_{i_j})$-solvable for all $1\le j\le t$.
        Let $\cA=\langle a_1,\ldots,a_m\rangle$ be the pebbling arrangement of $C'+D$, with $m=n+|C'|+|D|$.
        Denote by $C'_{\langle i,j\rangle}$ the portion of $C'$ that is in $\cA_{\langle i,j\rangle}$.
        For each $1\le j\le t$ define $\i_j$ such that $v_{i_j}=a_{\i_j}$, and define $\Clj=C'_{\langle 1,\i_j\rangle}$ and $\Crj=C'_{\langle \i_j,n\rangle}$, with $\Dlj=\{v_{i_1},\ldots,v_{i_j}\}$ and $\Drj=\{v_{i_j},\ldots,v_{i_t}\}$.  See an example in Figure \ref{f:ejemplo}.
        
        \begin{figure}
        \begin{center}
        \begin{floatrow}
        \begin{tikzpicture}[scale=1.5]
        \tikzstyle{every node}=[draw,circle,fill=black,minimum size=1pt,inner sep=2pt]
        \draw node (v1) [label=above: $v_1$] at (0,0) {};
        \draw node (v2) [label=above: $v_2$] at (1,0) {};
        \draw node (v3) [label=above: $v_3$] at (2,0) {};
        \draw node (v4) [label=above: $v_4$] at (3,0) {};
        \draw node (v5) [label=above: $v_5$] at (4,0) {};
        \draw node (v6) [label=above: $v_6$] at (5,0) {};
        \draw (v1) -- (v2) -- (v3) -- (v4) -- (v5) -- (v6);
        \draw (v1) to[out=60,in=120] (v3);
        \draw (v2) to[out=60,in=120] (v4);
        \draw (v3) to[out=60,in=120] (v5);
        \draw (v4) to[out=60,in=120] (v6);
        \draw plot[only marks,mark=x,mark size=2pt] coordinates {(0,-.2) (0,-.4) (2,-.2) (2,-.4) (3,-.2)};
        \draw node [fill=none] at (1,-.2) {};
        \draw node [fill=none] at (1,-.4) {};
        \draw node [fill=none] at (1,-.6) {};
        \draw node [fill=none] at (1,-.8) {};
        \draw node [fill=none] at (4,-.2) {};
        \draw node [fill=none] at (4,-.4) {};
        \draw node [fill=none] at (4,-.6) {};
        \draw node [fill=none] at (5,-.2) {};
        \end{tikzpicture}
        \hspace{0.5in}
        \begin{tabular}{cl}
        $\times$&$D=\{v_1^2,v_3^2,v_4\}$\\
        &\\
        $\ocircle$&$C=\{v_2^4,v_5^3,v_6\}$\\
        &\\
        &\\
        &\\
        &\\
        \end{tabular}
        \end{floatrow}
        \vspace{-0.5in}
        $$\cA(C+D)=\langle v_1, v_{1,1}, v_{1,2}, v_2, v_{2,1}, v_{2,2}, v_{2,3}, v_{2,4}, v_3, v_{3,1}, v_{3,2}, v_4, v_{4,1}, v_5, v_{5,1}, v_{5,2}, v_{5,3}, v_6, v_{6,1} \rangle$$
        \begin{tabular}{c|r}
            $j$&$\i_j$  \\
            \hline
            1&2\\
            2&3\\
            3&10\\
            4&11\\
            5&13\\
        \end{tabular}
        \hspace{0.5in}
        \begin{tabular}{l}
            $D'_{L_1}=\{v_1\}$\\
            \\
            $D'_{L_2}=\{v_1^2\}$\\
            \\
            $D'_{R_4}=\{v_3,v_4\}$
        \end{tabular}
        \hspace{0.5in}
        \begin{tabular}{l}
            $C'_{L_1}=\{\}$\\
            \\
            $C'_{R_2}=\{v_2^4,v_5^3,v_6\}$\\
            \\
            $C'_{L_4}=\{v_2^4\}$
        \end{tabular}
        \end{center}
        \caption{Example for $G=P_6^{(2)}$ with $t=5$, $v_{i_5}=v_4$, and $D$ does not precede $C$.} 
        \label{f:ejemplo}
        \end{figure}

        If some $1<j<t$ has $\Clj$  $\Dlj$-unsolvable and $\Crj$  $\Drj$-unsolvable then, by induction on $t$ there exist maximum $\Dlj$-unsolvable $\Cljs$ that is $\Dlj$-small, and maximum $\Drj$-unsolvable $\Crjs$ that is $\Drj$-small.
        By Corollary \ref{c:canonKn} and Lemma \ref{l:canonical}, we may choose both $\Cljs$ and $\Crjs$ to be canonical.
        Because $\dC'\cap\dD=\mt$, we have
        $s(D)<n$, which implies that either $s(\Dlj)<i_j$ or $s(\Drj)<n-i_j+1$.
        By symmetry, we will assume the former, from which follows that $\Cljs(v_{i_j})=0$.
        This implies that $\dC^*_{L_j}\cap \dC^*_{R_j}=\mt$.
%
%
        Then $C^*=\Cljs+\Crjs$ is  $D$-unsolvable and $D$-small,  since $\pot(C^*) = \pot(\Cljs)+\pot(\Crjs)=(j-1)+(t-j)=t-1$. 
        And it is of maximum size, since $|C^*| = |\Cljs|+|\Crjs| \ge |\Clj|+|\Crj| = |C'|$ and $C'$ is maximum $D$-unsolvable (the second equality holds because $C'(v_{i_j})=0$, by assumption).
       
        Thus, for all $1<j<t$, either $\Clj$ is $\Dlj$-solvable or $\Crj$ is $\Drj$-solvable.         
        The same holds true if $j\in\{1,t\}$ and $v_{i_j}$ is not simplicial,
        although induction is on $n$ instead of $t$.

        If $v_{i_1}$ is not simplicial, then $C'_{L_1}$ must solve $D_{L_1}$. Let $h$ be the biggest $j$ such that  $C'_{L_j}$  solves $D_{L_j}$. Clearly,  $h<t$ and   $C'_{R_{h+1}}$ does not solve $D_{R_{h+1}}$, which implies that either $h+1=t$ and $v_{i_t}=v_n$ is simplicial, or  $C'_{L_{h+1}}$ solves $D_{L_{h+1}}$ in contradiction to  the choice of $h$.  We conclude that if $v_{i_1}$ is not simplicial, then $v_{i_t}$ is. Therefore,  we may assume that the graph is labeled so that $v_1\in\dD$.
        
        First suppose that $c=0$ and set $C''=C'_{[k+1,n]}=C'$.
        For $G=\Pnk$, let $G'=G_{[k+1,n]}\cong P_{n-k}^{(k)}$; then $G'$ has diameter $d'=d-1$.
        In this case we use induction on $n$.
        We simplify notation somewhat; $D_L=D_{[1,k]}$, $D_R=D_{[k+1,n]}$, $t_L=|D_L|$, and $t_R=|D_R|=t-t_L$.
        Then set $n'=n-k$ and $t'=t+t_L$.
        
        If $G'$ is $t'$-wide then define $D'=D_R+2t_Lv_{k+1}$; then $|D'|=t'$.
        Note that $s(D')=s(D)-s(D_L)+\e$ for some $\e\in\{0,1\}$.
        Hence $s(D)-s(D')+2t_L = s(D_L)-\e+2t_L \le 3t_L \le 3t \le k$, so that $-s(D)\ge -k+2t_L-s(D')$.
        Therefore we have
        \begin{align*}
            |C''|
            &= |C'|\\
            &\geq n+2t-2-s(D)+1\\
            &\ge (n-k)+2(t+t_L)-2-s(D')+1\\
            &= n'+2t'-2-s(D')+1,
        \end{align*}
        and so $C''$ is $D'$-solvable in $G'$, which implies that $C'$ solves $D$ in $G$, a contradiction.
        
        If $G'$ is $t'$-long then define $D'=D_R+2D_L^{+k}$, where $D_L^{+k} = \sum_{i=1}^k D(v_i)v_{k+i}$ so that $s(D')=s(D)-s(D_L^{+k}\cap D_R)$.
        Now we have
        \begin{align*}
            |C''|
            &= |C'|\\
            &\ge n+2t-2-s(D)+1\\
            &= k(d-1)+b+2t-s(D)+1\\
            &\ge t(2^d-2)+b+2t-s(D)+1\\
            &\ge t2^d+b-s(D')+1 \\ 
            &= (t+t_L)2^{d-1} + t_R2^{d-1} + b - s(D') - s(D_L^{+k}\cap D_R) + 1\\
            &= [t'2^{d'} + b - s(D') + 1] + [t_R2^{d'} - s(D_L^{+k}\cap D_R)]\\
            &\ge [t'2^{d'} + b - s(D') + 1] + [4t_R - 3t_R]\\
            &\ge t'2^{d'} + b - s(D') + 1.
        \end{align*}
        Hence $C'$ is $D'$-solvable in $G'$, which implies that $C'$ solves $D$ in $G$, a contradiction.
        
        These contradictions imply that $c>0$.
        We now define $j$ to be minimum so that $v_j\in\dC'$, with $C''=C-v_j$ and $D'=D+v_j-v_1$.
        If $D(v_1)=1$ then $s(D')=s(D)$ and $D'(v_1)=0$, so set $G'=G-v_1\cong P_{n'}^{(k)}$, where $n'=n-1$.
        Then
        \begin{align*}
            |C''|
            &= |C'|-1\\
            &= (n+2t-2-s(D)+1)-1\\
            &= n'+2t-2-s(D')+1,
        \end{align*}
        which means, since $P_{n'}^{(k)}$ is $t$-wide, that $C''$ is $D'$-solvable on $G'$ by induction on $n$.
        If instead $D(v_1)>1$ then $s(D')=s(D)+1$ and $c'=|C''_{[i_t+1,k+1]}|=c-1$, and so
        \begin{align*}
            |C''|
            &= |C'|-1\\
            &= (n+2t-2-s(D)+1)-1\\
            &= n+2t-2-s(D')+1,
        \end{align*}
        which means that $C''$ is $D'$-solvable on $G$ by minimality of $c$.
        In either case, because $j\le k+1$, the pebble placed by $C''$ on $v_j$, along with the pebble of $C'$ already on $v_j$, can reach $v_1$, which makes $C'$ $D$-solvable, a contradiction.

        This final contradiction implies that $C'$ is $D$-small, which completes the proof that $\Pnk$ is $D$-small.

        We complete the proof of this subcase of Theorem \ref{t:Dsolve} as follows.
        Since $\Pnk$ is $D$-small, we have that $\p(\Pnk,D)=1+|C'|$ s.t. $C'$ is $D$-small.
        It is easy to see that if $C'$ is $D$-small, then $|C'|\le n-s(D)+2\pot(C')\le n-s(D)+2(t-1)$, and thus  $\p(\Pnk,D)\le 1+ n-s(D)+2(t-1)=|C|$, which implies the result.

    \end{enumerate}

    \item
    \label{tMid}
    \underline{\bf Intermediate Case: $t_0<t=\lceil t_0\rceil$.}
    \bigskip

    \begin{enumerate}[listparindent=1.5em]

        \item 
        \underline{\bf Subcase: $d=2$.}
        \bigskip\\
        We note that this case follows the same argument as in the $d=2$ subcase of the wide case $t\le t_0$; only the size of $C$ is different.
        
        Here we have $t=(k+1)/2$ for some odd $k$.
        Let $n=k+2+b$ and $|C|=4t+b-s+1$, where $s=s(D)$.
        Define $M=\{v_{b+2},\ldots,v_{k+1}\}$ and note that every vertex in $M$ is dominating.
        This means that every pair of potentials can solve any target via pebbling steps through any vertex in $M$.
        We first observe that $z\ge s$ , which implies that
        \begin{align*}
            \pot(C)
            &\ge \left \lceil \frac{(4t+b-s+1-n)+s}{2} \right \rceil\\
            &= t + \left \lceil \frac{2t+b+1-n}{2} \right \rceil\\
            &= t\ .
        \end{align*}
        Since $t\ge 2$ and any pair of potentials solves any root $r$ through any vertex of $M$, we have a solution  at cost $4=2^{\ecc(r)}$, 
        because any root has eccentricity 2.
        
        Next, for each $x\in\dD$, define $u(x)=2D(x)$ and, for each potential vertex $y$ of $C$, define $w(y)=2\lfloor C(y)/2\rfloor$ (i.e. twice the number of potentials at $y$).
        Then, since $t=(k+1)/2$, Corollary \ref{c:menger} (with $j=2t-1\leq k$)  yields $2t-1=k$ internally disjoint $(D,B)$-paths $\cP$, where $B$ is the set of potentials of $C$.
        Let $\r$ be the number of targets in $\dD$ that have a slide in $\cP$. 
        Then there are $t-\r$ targets in $\dD$, all of whose paths in $\cP$ contain at least one zero; that is, there are at least $2(t-\r)-1$ zeros of $C$ different from $\supp(D)$.
        Hence
        \begin{align*}
            \pot(C)
            &\ge \left\lceil\frac{(4t+b-s+1-n)+(s+2t-2\r-1)}{2}\right\rceil\\
            &= \left\lceil\frac{4t-2\r-1}{2}\right\rceil\\
            &= 2t-\r\\
            &= \r+2(t-\r)\ .
        \end{align*}
        Thus we can solve the remaining $t-\r$ targets via pairs of potentials, which solves $D$.\\

        \item 
        \underline{\bf Subcase: $d\ge 3$.}
        \bigskip\\
        Now we have $\ptnk = l_t(n,k) = t2^d+b = t2^d+(n-2)-k(d-1)$.
        Let $|C|=l_t(n,k)-s(D)+1$.
        
        Note that, because $t=\lceil t_0\rceil$, for any $n'=k(d'-1)+2+b$ for some $0\le b<k$, $P_{n'}^{(k)}$ is $t'$-wide for all $t'\le t$ and $d'=\diam(P_{n'}^{(k)})\le\diam(\Pnk)=d$, 
        provided that either $t'<t$ or $d'<d$:
        \[\left(\frac{d-1}{2^d-2}\right)\le \left(\frac{d'-1}{2^{d'}-2}\right),\]
        and thus 
        \[t'\le t = \lceil t_0(k,d)\rceil = \left\lceil k\left(\frac{d-1}{2^d-2}\right)\right\rceil \le \left\lceil k\left(\frac{d'-1}{2^{d'}-2}\right)\right\rceil = \lceil t_0(k,d')\rceil .\]
        Notice that we get a strict inequality if either $t'<t$ or $d'<d$, which would then imply that $t'\le t_0(k,d')$; i.e. $P_{n'}^{(k)}$ is $t'$-wide.
        
        Let $\cA=\langle a_1, \ldots, a_m\rangle$ be the pebbling arrangement for the pebbling function $C\vee D$.
        For $1\le i<j\le m$, 
        let $C_{\langle i,j\rangle}$ (resp., $D_{\langle i,j\rangle}$) refer to only those pebbles (resp., targets) of $(C\vee D)_{\langle i,j\rangle}$ that correspond to $C$ (resp., $D$).
        Now set $s_{\langle i,j\rangle}=|\supp(D_{\langle i,j\rangle})|$, let $h_j$ be such that $a_j\in V_{h_j}$, 
        and define 
        \[\D(j)=|C_{\langle 1,j\rangle}|-p_{|D_{\langle 1,j\rangle}|}(h_j,k)+s_{\langle 1,j\rangle}-1\]
        Notice that $1\le h_j \le n-b-1$ implies that
        \[\D(j)=|C_{\langle 1,j\rangle}|-2|D_{\langle 1,j\rangle}|-h_j+s_{\langle 1,j\rangle}+1,\]
        because $G_{[1,h_j]}$ is $|D_{\langle 1,j\rangle}|$-wide in that range.
        Similarly, the analogous formula holds for $G_{\langle j+1,m\rangle}$ 
        for $b \leq h_{j+1}\le n$ because it is $|D_{\langle j+1,m\rangle}|$-wide in that range.
        The function $\D$ serves as an indicator: $\D(j)\ge 0$ implies that $C_{\langle 1,j\rangle}$ solves $D_{\langle 1,j\rangle}$.
        If $\D(j)$ is too large, $|C_{\langle j+1,m\rangle}|$ may be too small to solve $D_{\langle j+1,m\rangle}$.
        However, if $\D(j)=\e\in\{0,1\}$ for some $j$ for which both $G_{\langle 1,j\rangle}$ is $|D_{\langle 1,j\rangle}|$-wide and $G_{\langle j+1,m\rangle}$ is $|D_{\langle j+1,m\rangle}|$-wide, then we can show that $C_{\langle j+1,m\rangle}$ solves $D_{\langle j+1,m\rangle}$.
        Indeed,
        \begin{align*}
            |C_{\langle j+1,m\rangle}|
            &= |C|-|C_{\langle 1,j\rangle}|\\
            &= (t2^d+b-s(D)+1)-(2|D_{\langle 1,j\rangle}|+h_j-s_{\langle 1,j\rangle}-1+\e)\\
            &\ge (2t+n-2-s(D)+2)-(2|D_{\langle 1,j\rangle}|+h_j-s_{\langle 1,j\rangle}-1+\e)\\
            &= 2(t-|D_{\langle 1,j\rangle}|) + (n-h_j) - (s(D)-s_{\langle 1,j\rangle})+(1-\e)\\
            &\ge 2|D_{\langle j+1,m\rangle}| + (n-h_j+1)-2 -s_{\langle j+1,n\rangle} + 1\\
            &\ge p_{|D_{\langle j+1,m\rangle}|}(G_{\langle j+1,m\rangle})-s_{\langle j+1,m\rangle}+1\\
            &\geq \pi(P_{n-h_{j-1}+1}^{(k)},D_{\langle j+1,m\rangle}).\
        \end{align*}

        Therefore, we aim to show that such a $j$ exists.        
        Let $j^*$ be maximum such that $G_{\langle 1,j*\rangle}$ is $|D_{\langle 1,j*\rangle}|$-wide.
        Then $G_{\langle 1,j\rangle}$ is $|D_{\langle 1,j\rangle}|$-wide for all $1\le j\le j^*$. 
        Thus
        \begin{equation}
            \label{e:delta1}    
            \D(j)-\D(j-1)=\left\{
            \begin{tabular}{ll}
                $-1$&for $a_j=v_h$, for some $h$,\\
                $+1$&for $a_j\in \dC$,\\
                $-1$&for $a_j\in \dD$ and $a_j=v_{h,1}$, for some $h$, and\\
                $-2$&for $a_j\in \dD$ and $a_j=v_{h,\ell}$, for some $h$ and some $\ell>1$.
            \end{tabular}\right.
        \end{equation}
        Moreover, for $j>j^*$, even in the cases for which $G_{\langle 1,j\rangle}$ is $|D_{\langle 1,j\rangle}|$-long, one can see that
        \begin{equation}
            \label{e:delta2}
            \D(j)-\D(j-1)>0\ \text{if and only if}\ a_j\in\dC,\ \text{in which case}\ \D(j)-\D(j-1)=1.
        \end{equation}
        Indeed, assume that $G_{\langle 1,j\rangle}$ is $t$-long.
        \begin{itemize}
            \item 
        Consider when $|D_{\langle 1,j-1\rangle}|=t$.
        If $G_{\langle 1,j-1\rangle}$ is $t$-long then $a_j\in V(G)\cup\dC$ since $D_{\langle j,m\rangle}$ is empty.
        In this case, 
        \begin{equation*}
            \label{e:delta3}    
            \D(j)-\D(j-1)=\left\{
            \begin{tabular}{ll}
                $-1$&for $a_j=v_h$, for some $h$, and\\
                $+1$&for $a_j\in \dC$.
            \end{tabular}\right.
        \end{equation*}
        If $G_{\langle 1,j-1\rangle}$ is $t$-wide then $a_j\in V_{(d-1)k+2}$, and so $\D(j)-\D(j-1)\le (d-1)k-t(2^d-2)-1\le -2$.
            \item
        Now consider when $|D_{\langle 1,j-1\rangle}|<t$.
        Then $G_{\langle 1,j-1\rangle}$ is $(t-1)$-wide, and so $a_j\in\dD$.
        In this case, 
        \begin{equation*}
            \label{e:delta4}    
            \D(j)-\D(j-1)\le \left\{
            \begin{tabular}{ll}
                $-2$&for $a_j\in \dD$ and $a_j=v_{h,1}$, for some $h$, and\\
                $-3$&for $a_j\in \dD$ and $a_j=v_{h,\ell}$, for some $h$ and some $\ell>1$.
            \end{tabular}\right.
        \end{equation*}
        \end{itemize}
        This proves (\ref{e:delta2}).
        
        Suppose that $D(v_1)>0$ and $D(v_n)>0$.
        Then $\D(2)=-1$ and $\D(m)=0$.
        Let $j'$ be minimum such that $\D(j')=0$, with $a_j\in V_h$.
        Because of Equations (\ref{e:delta1}) and (\ref{e:delta2}), we must have both $a_{j'-1}$ and $a_{j'}$ in $\dC$. 
        Indeed, $\D(j')=0$ implies that $\D(j'-1)=-1$ and $\D(j'-2)=-2$ by the definition of $j'$. This means that $\D(j')-\D(j'-1)=1$ and $\D(j'-1)-\D(j'-2)=1$, so that $a_{j'}$ and $a_{j'-1}$ are both pebbles; i.e. $v_h$ is big.
        Since no big vertex is adjacent to a target, we have $k+1<h<n-k$, which proves the existences of the desired $j$.
        
        Now suppose that $D(v_1)=0$ or $D(v_n)=0$.
        By symmetry we will assume that $D(v_1)=0$.
        Suppose that $v_h\in\dD$ for some $1<h\le k+1$.
        Because no big vertex is adjacent to a target vertex, $C(v_i)\le 1$ for all $i<h$; thus $|C_{[1,h-1]}|\le h-1$.
        Therefore 
        \begin{align*}
        |C_{[h,n]}|
        &= |C| - |C_{[1,h-1]}|\\
        &\ge (t2^d+n-2-k(d-1) -s(D)+1) - (h-1)\\
        &= t2^d + (n-h+1) -2 - k(d-1) -s(D)+1.
        \end{align*}
        If $h\le b$ then $P_{n-h+1}$ is $t$-long, and so this value equals $p_t(P_{n-h+1}^{(k)})-s(D)+1$, implying that $C_{[h,n]}$ solves $D$ by induction on $n$.
        If $b<h\le k+1$ then $P_{n-h+1}$ is $t$-wide, and so this value is at least $(n-h+1)+2t-2-s(D)+1 = p_t(P_{n-h+1}^{(k)})-s(D)+1$, implying that $C_{[h,n]}$ solves $D$ by induction on $n$.

        Thus we may assume that $D_{[1,k+1]}$ is empty.
        Notice that the above arithmetic also proves that if $h$ is the minimum index such that $v_h\in\dD$, then $|C_{[1,h-1]}|\ge h$: if $|C_{[1,h-1]}|\le h-1$ then remove $G_{[1,h-1]}$; then $C_{[h,n]}$ solves $D$ by induction on $n$, because $G_{[h,n]}$ is $t$-wide, since $h>k+1$.
        We will make use of this below.
        
        Let $j_0$ be minimum so that $a_{j_0}\in\dD$; thus $a_{j_0}\in V_h$.
        Because $|C_{\langle 1,j_0\rangle}| = |C_{[1,h-1]}|\ge h\ge \p(G_{\langle 1,j_0\rangle},v_h)$ ($G_{\langle 1,j_0\rangle}$ is 1-wide because $1<t$), we have $\D(j_0)\ge 0$.
        If $\D(j_0)\in\{0,1\}$ then we are done; thus we will assume that $\D(j_0)>1$.
        Now let $j^*$ be maximum so that $a_{j^*}\in\dD$ and let $h^*$ be such that $a_{j^*}\in V_{h^*}$.
        Of course, if $\D(j^*)\ge 0$ then we are done, so assume otherwise; i.e. $\D(j^*)<0$.
        
        By the definition of $j^*$, we have that $G_{\langle 1,j\rangle}$ is $|D_{\langle 1,j\rangle}|$-wide for all $1\le j<j^*$.
        Thus Equation (\ref{e:delta1}) applies.
        Because $\D$ never decreases by more than 2, this implies the existence of some $j_0\le j'<j^*$ with $\D(j')\in\{0,1\}$, completing the proof in this case.
    \end{enumerate}

        \item 
        \label{tLong}
        \underline{\bf Long Case: $t>\lceil t_0\rceil$.}
        \bigskip\\
        Here we have $\ptnk=l_t(n,k)=t2^d+(n-2)-k(d-1)$.
        Let $|C|=l_t(n,k)-s(D)+1$.
        In this case, $t-1\ge\lceil t_0\rceil$, and so $(t-1)(2^d-2)\ge k(d-1)$, which implies that $t2^d-k(d-1)\ge 2^d-2+2t\ge 2^d-d+1$ since $2t+d\ge 3$.
        Hence
        \begin{equation}
            \label{e:cheapcompare}
            t2^d+n-2-k(d-1)\ge 2^d+n-d-1.
        \end{equation}

        Let $T_i$ be a breadth-first-search spanning tree of $G$, rooted at $v_i$ (with $T_1$ isomorphic to a caterpillar with main path of length $d$), then $\p(T_i,v_i)\le \p(T_1,v_1)=2^d+n-d-1$.
        Then inequality (\ref{e:cheapcompare}) implies that $|C|\ge \p(T_1,v_1)-s(D)+1\ge \p(T_i,v_i)-s(D)+1$ and so, by Theorem \ref{t:tree}, $C$ solves any target $v_i\in\dD$ along $T_i$ with cost at most $2^{\ecc_{T_i}(v_i)}$.

        If some $x\in\dD$ has $D(x)\ge 2$ or has $\ecc_{T_i}(x)<d$, then
        \begin{equation}
        \label{e:cheap}
            \begin{tabular}{lcl}
            $|C-C[\s_x]|$
            &$\ge$& $(p_t(n,k)-s(D)+1)-2^{\ecc_{T_i}(x)}$\\
            &$=$& $(p_{t-1}(n,k)-s(D-x)+1)+(2^d-2^{\ecc_{T_i}(x)})+(s(D-x)-s(D))$\\
            &$\ge$& $\p(G,D-x)$,
            \end{tabular}
        \end{equation}
        and so $C-\s_x$ solves $D-x$.

        Hence $D$ consists of singleton targets only, each with eccentricity $d$.
        That is, $D=D_{[1,b+1]}\cup D_{[n-b,n]}$; i.e. $\dD\ \sse\ V_{[1,b+1]}\cup V_{[n-b,n]}$.
        Without loss of generality, we may assume that $\dD_{[1,b+1]}\not=\mt$.
        In addition, if $\cost(\s_{v_i})=2^d$ for $i\in [1,b+1]$ then every potential vertex $v_j$ in $C[\s_{v_i}]$ is at distance $d$ from $v_i$; that is, $j\in [n-b+i-1,n]$.
        Since no big vertex is adjacent to a target, we then have $\dD_{[n-b,n]}=\mt$; i.e. $D=D_{[1,b+1]}$.
        Furthermore, if any length-$d$ path $P$ from a {\it huge} ($C(v)\ge 2^{\dist(v,r)}$ for some $r\in\dD$) vertex $v$ to some target $r$ contains an interior pebble, then there is a super-cheap $r$-solution from $v$ along $P$.
        The calculations of Inequality (\ref{e:cheap}) show that the existence of a cheap solution yields a $t$-fold solution by induction on $t$.

        Now, if $v_j$ is big, then $j\in [n-b+i-1,n]$, so define $C'$ by $C'(v_j)=C(v_j)-2$, $C'(v_n)=C(v_n)+2$, and $C'(v)=C(v)$ otherwise.
        Then any $(C',D)$-solution can be converted to a $(C,D)$-solution because $N(v_n)\sse N[v_j]$.
        Thus, if $C$ is not $D$-solvable, then neither is $C'$.
        In other words, we may assume that the only big vertex of $C$ is $v_n$.

        This implies that, for every $v_i\in\dD$, the {\it rail} $R_i=(v_i,v_{i+k},\ldots,v_{i+(d-1)k})$ 
        has no pebbles.
        Let $j$ be minimum such that $v_j\in\dD$.
        Now define the graph $G'=G-\{R_i\}_{v_i\in\dD-v_j}$, with corresponding configuration $C'=C_{G'}$, and notice that $G'\cong P_{n'}^{(k')}$, where $n'=n-(t-1)d$ and $k'=k-(t-1)$.
        In addition, $b'=b-(t-1)$ and $d'=d$.
        By induction on $n$, we have that
        \begin{align*}
                |C'|
            &= |C|\\
            &= (t2^d+b)-s(D)+1\\
            &= t2^{d'}+b'\\
            &= p_t(n',k').
        \end{align*}
        which implies that $C'$ is $t$-fold $v_j$-solvable on $G'$.
        Equivalently, $C$ is $t$-fold $v_j$-solvable on $G$, using only the vertices of $G'$.
        Because $G_{[1,k+1]}$ is a clique, any step $v_{i'}\peb v_j$ can be converted instead to $v_{i'}\peb v_{j'}$ for any $v_{j'}\in\dD$, thus solving $D$.
\end{enumerate}
This completes the proof.
\end{proof}

\section{Proof of Theorem \ref{t:main}}
\label{s:thm}

\begin{lem} 
\label{l:down_simpl}
If $r$ is a simplicial vertex of $\Pnk$, then $\pi_t(\Pnk,r)\geq \ptnk$.
\end{lem}

\begin{proof}
Let $r$ be the simplicial vertex $v_1$ of $P^{(k)}_n$. 
For $t\ge 1$ we define the configurations $W=\Wtn$  and $L=\Ltnk$ on $P^{(k)}_n$ with target $r$ as follows. 
\vspace{0.2in}

\begin{center}

$W(v_i)=\left\{ \begin{array}{rl}
0,  & \hbox{ for } i=1;\\
1, & \hbox{ for } 2\leq i \leq n-1;\\
 2t-1, & \hbox{ for } i=n.  
 \end{array}
\right.$
\hspace{0.2in}
$L(v_i)=\left\{
  \begin{array}{rl}
0, & \hbox{ for } 1\leq i\leq (d-1)k+1;\\
    1, & \hbox{ for }(d-1)k+1< i \leq n-1;\\
    t2^{d}-1, & \hbox{ for } i=n.  \end{array}
\right.$
\end{center}
\vspace{0.2in}

Notice that $|W|=\wtn-1$ and $|L|=\ltnk-1$.
Let $V_i$ denote the vertices at distance $i$ from $v_1$.

Since $\pot(W)=t-1$, Fact \ref{f:pot} shows that $W$ is $t$-fold  $r$-unsolvable.

Since $\pot(L)=t2^{d-1}-1$, $L$ can only place at most $t2^{d-1}-1$ pebbles on $V_{d-1}$.
If $d=1$ then $\pot(L)=t-1$, so Fact \ref{f:pot} shows that $L$ is $t$-fold $r$-unsolvable.
If $d>1$ then, by induction on $d$, $L$ can only place at most $t2^{d-j}-1$ pebbles on $V_{d-j}$ for each $1\le j\le d$.
Thus $L$ can only place at most $t-1$ pebbles on $r$, showing that $L$ is $t$-fold $r$-unsolvable.
\end{proof}

\noindent
{\it Proof of Theorem \ref{t:main}.}
Lemma \ref{l:down_simpl} provides the lower bound.
Theorem \ref{t:Dsolve} provides the upper bound.
\hfill\pf

\section{Proof of Corollary \ref{c:expo}} 
\label{s:Correct}

\begin{proof}
The values for $n\le 8$ are easy to check, so we assume that $n\ge 9$.
We know that $k$ and $d$ must be related to $n$ by $k(d-1)=n-2-b$, for some $0\le b<k$; that is,
\begin{equation}
\label{e:one}
    n-2-k < k(d-1) \le n-2.
\end{equation}
Then $\ep(P_n)$ is the least value of $k$ for which this relation holds with
\begin{equation}
\label{e:two}
    k\ge (2^d-2)/(d-1).
\end{equation}

Given $n\ge 9$, define $\k = M(n)+1 = \lfloor (n-2)/(\lceil\lg n\rceil -2)\rfloor+1$ and $\d=\lfloor (n-2)/\k\rfloor +1$, so that 
when $k=\k$ we have 
$d=\d$.
Thus the upper bound follows by showing that (\ref{e:two}) also holds with these values of $k$ and $d$.

We first observe that the definition of $\k$ implies that
\begin{equation}
    n-2 = (\k-1)(\lceil\lg n\rceil-2)+\ell,
\end{equation}
for some $0\le \ell<\lceil\lg n\rceil-2$.
Define $\ol=(\lceil\lg n\rceil -2)-\ell$; then
\[\k = \frac{n-2-\ell}{\lceil\lg n\rceil-2}+1 = \frac{(n-2)+\ol}{\lceil\lg n\rceil-2},\]
so that
\begin{align*}
    \frac{n-2}{\k}
    &= \frac{(n-2)(\lceil\lg n\rceil-2)}{(n-2)+\ol}\\
    &= (\lceil\lg n\rceil -2) - \frac{\ol(\lceil\lg n\rceil-2)}{(n-2)+\ol},
\end{align*}
and
\[\lceil\lg n\rceil - 3
\le (\lceil\lg n\rceil -2) - \frac{(\lceil\lg n\rceil -2)^2}{n}
\le (\lceil\lg n\rceil -2) - \frac{\ol(\lceil\lg n\rceil-2)}{(n-2)+\ol}
< \lceil\lg n\rceil -2.
\]
Hence $\d-1 = \lfloor (n-2)/\k\rfloor = \lceil\lg n\rceil -3$.
This yields $2^\d = 2^{\lceil\lg n\rceil-2}\le 2^{\lg n -1} = n/2$.
Therefore, when $n\ge 9$ we have
\begin{align*}
    \k \left\lfloor\frac{n-2}{\k}\right\rfloor
    &= \left(\left\lfloor\frac{n-2}{\lceil\lg n\rceil -2}\right\rfloor +1\right)\Big(\lceil\lg n\rceil -3\Big)\\
    &> \left(\frac{n-2}{\lceil\lg n\rceil -2}\right)\Big(\lceil\lg n\rceil -3\Big)\\
    &\ge (n-4)/2\\
    &\ge 2^\d-2,
\end{align*}
finishing the proof of the upper bound.

For the lower bound, the values for $n\le 2^{10}$ can be checked by computer, so we assume that $n>2^{10}$.
Given $n$, we define $\l=\lceil\lg n\rceil-2\ge 9$, $N=n-2^{\l+1}$, $M=\lfloor (n-2)/\l\rfloor$, and $L=(n-2)\mod\l$. 
Then $n-2=M\l+L$, $0<N\le 2^{\l+1}$ and $0\le L<\l$.
Furthermore, we define $m=\lfloor (n-2)/\l^2\rfloor$, and $\ell=(n-2)\mod\l^2$, so that $n-2=m\l^2+\ell$, with $0\le \ell<\l^2$.
Finally, define $\k'=M-m-1$ and $\d'=\lfloor (n-2)/\k'\rfloor+1$, so that (\ref{e:one}) holds when $k=\k'$ and $d=\d'$.
Thus the lower bound follows by showing that (\ref{e:two}) fails with these values of $k$ and $d$.

We begin by rewriting $\k'$ as
\begin{align*}
    \k'
    &= \frac{2^{\l+1}-2-L+N}{\l} - \frac{2^{\l+1}-2-\ell+N}{\l^2} - 1\\
    &= \frac{\l(2^{\l+1}-2-L+N) - (2^{\l+1}-2-\ell+N) - \l^2}{\l^2}.
\end{align*}
Then we show that $\d'-1=\l+1$ as follows.
For the upper bound we use the fact that $\ell-\l L\ge -(\l-1)^2$, while for the lower bound we use $\ell-\l L\le \l(\l-1)$.
Additionally, we use the following inequality, which can be easily checked to hold for $\l\ge 6$:
\begin{equation}
    \label{e:three}
    \frac{\l^2(2^{\l+1}-2)}{(\l-1)(2^{\l+1}-2)-(\l-1)^2-\l^2} < \l+2.
\end{equation}
Then
\begin{align*}
    \d'-1
    &= \left\lfloor\frac{\l^2(2^{\l+1}-2+N)}{\l(2^{\l+1}-2-L+N) - (2^{\l+1}-2-\ell+N) - \l^2}\right\rfloor\\
    &= \left\lfloor\frac{\l^2(2^{\l+1}-2+N)}{(\l-1)(2^{\l+1}-2+N) + (\ell-\l L) - \l^2}\right\rfloor\\
    &\le \left\lfloor\frac{\l^2(2^{\l+1}-2+N)}{(\l-1)(2^{\l+1}-2+N) - (\l-1)^2 - \l^2}\right\rfloor\\
    &\le \left\lfloor\frac{\l^2(2^{\l+1}-2)}{(\l-1)(2^{\l+1}-2) - (\l-1)^2-\l^2}\right\rfloor\\
    &\le \l+1,
\end{align*}
the last inequality because of (\ref{e:three}).
Also
\begin{align*}
    \d'-1
    &= \left\lfloor\frac{\l^2(2^{\l+1}-2+N)}{(\l-1)(2^{\l+1}-2+N) + (\ell-\l L) - \l^2}\right\rfloor\\
    &\ge \left\lfloor\frac{\l^2(2^{\l+1}-2+N)}{(\l-1)(2^{\l+1}-2+N) - \l}\right\rfloor\\
    &\ge \left\lfloor\frac{\l^2}{(\l-1)}\right\rfloor\\
    &\ge \left\lfloor\frac{\l^2-1}{(\l-1)}\right\rfloor\\
    &= \l+1.
\end{align*}

Now we show that inequality (\ref{e:two}) fails when $k=\k$ and $d=\d$.
To do this, recall from above that $N\le 2^{\l+1}$ and $l-\l L\le \l(\l-1)$.
Also, we note that $2\l^3+\l^2-\l+2<2^{\l+2}$ for $\l\ge 9$.
Then
\begin{align*}
    \k'(\d'-1)
    &= \left(\frac{\l(2^{\l+1}-2-L+N) - (2^{\l+1}-2-\ell+N) - \l^2}{\l^2}\right)(\l+1)\\
    &= \frac{(\l^2-1)(2^{\l+1}-2+N)+(\l+1)(\ell-\l L+\l^2)}{\l^2}\\
    &\le \frac{(\l^2-1)(2^{\l+2}-2)+(\l+1)(\l(\l-1)+\l^2)}{\l^2}\\
    &= 2^{\l+2}-2 - \frac{(2^{\l+2}-2) - (2\l^3+\l^2-\l)}{\l^2}\\
    &= 2^{\l+2}-2 - \frac{2^{\l+2} - (2\l^3+\l^2-\l+2)}{\l^2}\\
    &< 2^{\l+2}-2\\
    &= 2^{\d'}-2.
\end{align*}

This finishes the proof of the lower bound and, hence, the corollary.
\end{proof}

\section{Comments}  
\label{s:comments}

We observe the following simple corollary.

\begin{cor}
\label{c:WideExact}
Suppose that $\Pnk$ is $t$-wide, $d=\diam(\Pnk)$, and $D$ is a distribution of size $t$.
Then
\[\p(\Pnk,D)= \left\{
\begin{array}{ll}
    2t-\min_{v} D(v) & \hbox{if } s(D)=n, \hbox{and}\\
    n+2t-1-s(D)  & \hbox{if } s(D)<n.
\end{array}\right.\]
\end{cor}

\begin{proof}
The $s(D)=n$ case follows from Fact \ref{f:WideSupport} and Theorem \ref{t:piKn}, while the $s(D)<n$ case follows from Theorem \ref{t:DsolveStrong} and Lemma \ref{l:canonical}.
\end{proof}

We finish with a few comments, conjectures, and open problems that are driven by this work.

\begin{enumerate}
    \item
    It would be interesting to calculate a precise formula for $\p(\Pnk,D)$ when $\Pnk$ is $t$-long.
    \item 
    Does the above proof yields an efficient algorithm for $D$-solving pebbling configurations of size at least $\p_t(\Pnk)-s(D)+1$?
    \item
    Is there a formula for $\p(T,D)$ for any distribution $D$ on any tree $T$ (akin to the $|D|=1$ case)?
    If so, what is it and how is it constructed?
    \item
    We conjecture that the pebbling number of a chordal graph is always witnessed at a simplicial vertex; that is, if $G$ is chordal then there is some simplicial $r$ such that $\p(G)=\p(G,r).$
    \item
    What is the right generalization of Theorem \ref{t:Simplicial} for a general target $D$ on a chordal graph?
    \item
    Note that $\p(P_{d+1} \Box K_k) = k2^d$ was proved in \cite{Chung}.
    For $n=k(d+1)$, we have $P_{d+1} \Box K_k \subset \Pnk \subset P_d \boxtimes K_k $, so that 
    \[\p_t(P_{d+1} \boxtimes K_k) \le \p_t(\Pnk) \le \p_t(P_{d+1} \Box K_k) \le t\p(P_{d+1} \Box K_k) = tk2^d.\]
    One can view $P_{d+1}\boxtimes K_k$ as a ``blow-up'' of the path $P_d$ by cliques $K_k$: each vertex of $P_d$ is replaced by a clique of size $k$, and vertices from consecutive cliques are adjacent.
    More generally we could blow up the path vertices by different amounts, using the notation $K_{k_0}\ \rotatebox[origin=c]{90}{$\Join$}\ K_{k_2}\ \rotatebox[origin=c]{90}{$\Join$}\ \cdots\ \rotatebox[origin=c]{90}{$\Join$}\ K_{k_d}$, which is what is studied by Sieben in \cite{Sieben}.  
    In this context, for $n=k(d-1)+b+2$ and $G = K_{1}\ \rotatebox[origin=c]{90}{$\Join$}\ K_{k}\ \rotatebox[origin=c]{90}{$\Join$}\ \cdots\ \rotatebox[origin=c]{90}{$\Join$}\ K_{k}\ \rotatebox[origin=c]{90}{$\Join$}\ K_{b+1}$, we have $\Pnk\subset G$, so that $\p(\Pnk) = \p(G)$, since the arguments of Lemma \ref{l:down_simpl} work on $G$ as well.
    \item
    Might some of the methods developed here be useful in lowering the upper bound of $k(d)\le 2^{2d+3}$ from \cite{CHKT} mentioned in the Introduction?
\end{enumerate}


\end{document}